\documentclass[11pt]{article}

\usepackage[T1]{fontenc}
\usepackage[utf8]{inputenc}
\usepackage{lmodern}

\usepackage{amsmath, amssymb, amsthm}
\usepackage{mathrsfs}

\usepackage[margin=2.5cm]{geometry}
\usepackage{microtype}

\usepackage{xcolor, graphicx, float}
\usepackage[breakable]{tcolorbox}
\usepackage[shortlabels]{enumitem}

\usepackage{algorithm}
\usepackage{algorithmic}
\numberwithin{algorithm}{section}

\usepackage{etoolbox}

\makeatletter
\patchcmd{\@maketitle}{\LARGE \@title}{\LARGE\bfseries\@title}{}{}

\renewcommand{\@seccntformat}[1]{\csname the#1\endcsname.\quad}

\renewcommand\paragraph{\@startsection{paragraph}{4}{\z@}%
    {1.5ex \@plus1ex \@minus.2ex}%
    {-1em}%
    {\normalfont\normalsize\bfseries}}
\makeatother

\makeatletter
\def\th@plain{%
	\thm@notefont{}
	\itshape 
}
\def\th@definition{%
	\thm@notefont{}
	\normalfont 
}

\renewenvironment{proof}[1][\proofname]{\par
	\normalfont
	\topsep0\p@\@plus3\p@ \trivlist
	\item[\hskip\labelsep\itshape
	#1\@addpunct{.}]\ignorespaces
}{%
	\qed\endtrivlist
}
\makeatother

\newtheorem{theorem}{Theorem}[section]
\newtheorem{lemma}[theorem]{Lemma}

\theoremstyle{definition}
\newtheorem{definition}[theorem]{Definition}
\newtheorem{example}[theorem]{Example}
\newtheorem{remark}[theorem]{Remark}
\newtheorem{assumption}[theorem]{Assumption}

\definecolor{darkblue}{rgb}{0,0,.5}

\usepackage{hyperref}
\hypersetup{
	colorlinks=true,		
	linkcolor=darkblue,		
	citecolor=darkblue,		
	urlcolor=darkblue 		
}

\usepackage[capitalize,nameinlink]{cleveref}
\crefname{section}{section}{sections}
\crefname{subsection}{subsection}{subsections}
\Crefname{section}{Section}{Sections}
\Crefname{subsection}{Subsection}{Subsections}

\Crefname{figure}{Figure}{Figures}

\crefformat{equation}{\textup{#2(#1)#3}}
\crefrangeformat{equation}{\textup{#3(#1)#4--#5(#2)#6}}
\crefmultiformat{equation}{\textup{#2(#1)#3}}{ and \textup{#2(#1)#3}}
{, \textup{#2(#1)#3}}{, and \textup{#2(#1)#3}}
\crefrangemultiformat{equation}{\textup{#3(#1)#4--#5(#2)#6}}%
{ and \textup{#3(#1)#4--#5(#2)#6}}{, \textup{#3(#1)#4--#5(#2)#6}}{, and \textup{#3(#1)#4--#5(#2)#6}}

\Crefformat{equation}{#2Equation~\textup{(#1)}#3}
\Crefrangeformat{equation}{Equations~\textup{#3(#1)#4--#5(#2)#6}}
\Crefmultiformat{equation}{Equations~\textup{#2(#1)#3}}{ and \textup{#2(#1)#3}}
{, \textup{#2(#1)#3}}{, and \textup{#2(#1)#3}}
\Crefrangemultiformat{equation}{Equations~\textup{#3(#1)#4--#5(#2)#6}}%
{ and \textup{#3(#1)#4--#5(#2)#6}}{, \textup{#3(#1)#4--#5(#2)#6}}{, and \textup{#3(#1)#4--#5(#2)#6}}

\crefdefaultlabelformat{#2\textup{#1}#3}

\AtBeginDocument{
  \setlength{\abovedisplayskip}{6pt plus 3pt minus 3pt}
  \setlength{\belowdisplayskip}{6pt plus 3pt minus 3pt}
  \setlength{\abovedisplayshortskip}{0pt plus 3pt}
  \setlength{\belowdisplayshortskip}{3pt plus 3pt minus 3pt}
  \setlength{\jot}{2pt}
}

\allowdisplaybreaks

\DeclareMathOperator{\dom}{dom}
\DeclareMathOperator{\Image}{Im}
\DeclareMathOperator{\ri}{ri}
\DeclareMathOperator{\crit}{crit}
\DeclareMathOperator{\dist}{dist}
\DeclareMathOperator{\err}{err}

\begin{document}

\title{A Bregman proximal linearized ADMM for fractional programming with nonlinear coupling constraints\thanks{This work is supported by the National Natural Science Foundation of China (12601615 and 72471227), the Australian Research Council (DP230101749), and the China Postdoctoral Science Foundation (2026M793421).}}

\author{Pengjie Liu\thanks{School of Mathematics, Jiangsu Center for Applied Mathematics (CUMT), China University of Mining and Technology, Xuzhou 221116, China (\url{liupengjie2019@163.com}, \url{chenjh0516@163.com}, \url{shaohu@cumt.edu.cn)}.},
Jinhu Chen\footnotemark[2],  
Minh N. Dao\thanks{School of Science, RMIT University, Melbourne, VIC, 3000, Australia \url{(minh.dao@rmit.edu.au)}.},  
Hu Shao\footnotemark[2], 
and Jinbao Jian\thanks{School of Mathematical Sciences, Center for Applied Mathematics of Guangxi, Guangxi Minzu University, Nanning 530006, China \url{(jianjb@gxu.edu.cn)}.}}

\date{October 8, 2026}

\maketitle

\begin{abstract}
This paper considers a nonconvex fractional optimization problem with composite structure and nonlinear equality constraints. We reformulate the problem as a non-fractional min-max problem with corresponding optimal solutions. By linearizing the nonlinear constraint terms in the augmented Lagrangian and incorporating Bregman proximal linearization and a relaxation factor in $(0, 2)$ into the alternating direction method of multipliers (ADMM), we develop a Bregman proximal linearized ADMM with guaranteed subsequential convergence to a lifted critical point. Convergence of the entire sequence is established under the Kurdyka--{\L}ojasiewicz (KL) property. We further derive convergence rates under either a H\"olderian value proximity error bound condition or the KL property, with the corresponding exponent in $[0, 1)$. In particular, we establish superlinear convergence for exponents in $(0, 1/2)$, together with finite, linear, and sublinear convergence in the remaining cases. Numerical experiments demonstrate the effectiveness of the proposed algorithm.
\end{abstract}


\noindent{\bf Mathematics Subject Classification (MSC 2020):}
90C32,  
90C26,  
49M27,  
65K05.  

\section{Introduction}\label{sec:introdection}	

In this paper, we consider the following nonconvex composite-structured fractional optimization problem:
\begin{align}\label{Equ:1}
\min \Bigl\{ f(u) + \frac{h(x)}{g(x)}: F(x) = M u, \, x \in \mathcal{S}, \, u \in \mathbb{R}^m \Bigr\},
\end{align}
where $\mathcal{S}$ is a nonempty closed subset of $\mathbb{R}^n$, $f:\mathbb{R}^m \to \mathbb{R}$ is a continuously differentiable (possibly nonconvex) function, $h:\mathbb{R}^n \to (-\infty, +\infty]$ is a proper lower semicontinuous (l.s.c.) convex (possibly nonsmooth) function that is nonnegative and Lipschitz continuous on an open neighborhood of $\mathcal{S}$, $g:\mathbb{R}^n \to (-\infty, +\infty]$ is a proper l.s.c. convex (possibly nonsmooth) function that is positive on $\mathcal{S}$ and Lipschitz continuous on an open neighborhood of $\mathcal{S}$, $F:\mathbb{R}^n \to \mathbb{R}^q$ is a continuously differentiable mapping defined by $F(x) := (F_1(x), F_2(x), \dots, F_q(x))$, and $M \in \mathbb{R}^{q \times m}$ is a full column-rank matrix. The formulation in problem \cref{Equ:1} encompasses a wide range of optimization problems arising in engineering applications, including sparse signal recovery \cite{rahimi2019scale,sra2011optimization,studer2014stable,tao2022minimization} and image reconstruction \cite{ji2012image,repetti2014euclid}. When $g(x) \equiv 1$, problem \cref{Equ:1} reduces to a nonconvex nonlinear composite minimization problem \cite{bolte2018nonconvex,hien2024multiblock}. We present two examples to illustrate the generality of problem~\cref{Equ:1}.

\begin{example}[Sparse signal recovery \cite{rahimi2019scale}]
This model has attracted considerable attention in recent years and can be formulated as
\begin{align}\label{0127eq:2}
\min \left\{ \frac{\|x\|_1}{\|x\|_2}: Ax = b,\; \underline{x} \leq x \leq \overline{x},\; x \in \mathbb{R}^n \right\},
\end{align}
where $\underline{x}, \overline{x} \in \mathbb{R}^n$ denote the lower and upper bounds of the underlying signal, respectively.
A penalized formulation of \eqref{0127eq:2} is often considered:
\begin{align*}
\min \left\{ \lambda \frac{\|x\|_1}{\|x\|_2} + \frac{1}{2}\|Ax - b\|_2^2: \underline{x} \leq x \leq \overline{x},\; x \in \mathbb{R}^n \right\},
\end{align*}
where $\lambda > 0$ is a penalty parameter.
It is straightforward to verify that problem~\eqref{Equ:1} reduces to the above penalized formulation when
$F(x) = Ax - b$, $M = I$, $f(u) = \frac{1}{2}\|u\|_2^2$, $h(x) = \lambda \|x\|_1$,
$g(x) = \|x\|_2$, and
$\mathcal{S} = \{ x \in \mathbb{R}^n: \underline{x} \leq x \leq \overline{x} \}$.
\end{example}

\begin{example}\label{0425example:1} 
In neural network research, classification is one of the core tasks. A simple classification network can be formulated as
\begin{align*}
\min \Bigg\{  -\frac{1}{N} \sum_{i=1}^{N}\sum_{j=1}^K Y_{ij}\log(u_{ij}) &+ \gamma R(x): \\
\sigma(X^\top v_i+\hat{b}_i) =& u_i, \;x \in \mathcal{S}, \, u_i \in \mathbb{R}^{K}, \, i = 1, 2,\dots, N \Bigg\},
\end{align*}
where $N$ denotes the number of samples, $K$ denotes the number of classes and $v_i\in \mathbb{R}^d$ denotes the feature vector of the $i$-th sample.
Let $x=(x_1^\top,x_2^\top,\dots,x_K^\top,b^\top)^\top$, where $X=[x_1,x_2,\dots,x_K]\in \mathbb R^{d\times K}$ and $b\in \mathbb R^K$ is the class-wise bias vector. Denote $u=(u_1^\top,\dots,u_N^\top)^\top$, where $u_i\in \mathbb R^K$. The nonlinear constraint is $\sigma(X^\top v_i+b)=u_i$. In addition, $\sigma:\mathbb R^K\to\mathbb R^K$ is a smooth activation mapping and $Y_i\in \mathbb R^K$ is the one-hot label vector corresponding to the $i$-th sample.
We have $Y_{ij}=1$ if sample $i$ belongs to class $j$, and $Y_{ij}=0$ otherwise.
Moreover, $R(x)$ is a nonsmooth regularization term, and $\gamma$ is the regularization parameter.
In particular, with the cross-entropy loss $f(u)=-\frac{1}{N} \sum_{i=1}^{N}\sum_{j=1}^K Y_{ij}\log(u_{ij})$,
by setting $R(x) = \|x\|_1 / \|x\|_2$, defining $F_i(x) = \sigma(X^{\top} v_i+b)$, and taking $M = \boldsymbol{I}$ and $\mathcal{S}\subseteq \{x\in \mathbb{R}^n:\|x\|_2>0\}$,
the above formulation can be regarded as a special case of problem \cref{Equ:1}.
\end{example}

\subsection{Related work}

When $ f(u) \equiv 0 $ and no constraint conditions are considered, problem \cref{Equ:1} reduces to the latter minimization problem $ \min_{x\in \mathcal{S}}\{ h(x)/g(x) \} $. Common methods for such problems include the Dinkelbach method and its variants \cite{crouzeix1985algorithm,dinkelbach1967nonlinear,ibaraki1983parametric,Pang1980oarametric}. In the Dinkelbach method, at the $(k+1)$-th iteration, defining $ \theta_{k} = h(x_k)/g(x_{k}) $, one needs to solve the following auxiliary problem:
\begin{align}\label{eq:6}
x_{k+1} = \underset{x\in \mathcal{S}}{\arg\min} \left\{ h(x) - \theta_k g(x) \right\}.
\end{align}
However, solving the auxiliary problem \cref{eq:6} is not less computationally demanding than solving the original problem. To address the issue, proximal-type methods rooted in Dinkelbach's approach have been developed to tackle single-ratio fractional programming problems \cite{boct2017proximal,boct2021extrapolated,zhang2022first}. Bo\c{t} et al. \cite{boct2017proximal} proposed a proximal gradient-type method for fractional programming, where the numerator is proper, l.s.c., and convex, and the denominator is smooth (either concave or convex). In other words, both the numerator and the denominator demand favorable mathematical properties of the involved functions. Bo\c{t} et al. \cite{boct2021extrapolated} further extended the results of \cite{boct2017proximal} by relaxing both the numerator and denominator to possibly nonconvex and nonsmooth cases and introducing extrapolation steps in each iteration. Zhang et al. \cite{zhang2022first} further relaxed the convexity assumption on the numerator and developed related methods with monotone and nonmonotone line searches. However, when $ f(u) \not\equiv 0 $, the Dinkelbach method becomes inapplicable. In particular, when $ M = I $ and $ F(x) = x $, problem \cref{Equ:1} can be transformed into  
\begin{align}\label{Eq:n1}
\min\left\{ f(x) + \frac{h(x)}{g(x)}: \;x\in \mathcal{S} \right\}.
\end{align}
By applying the Dinkelbach method, one obtains  
$ x_{k+1} ={\arg\min}_{x\in \mathcal{S}} \{ h(x) + g(x)f(x) - \theta_k g(x) \} $,  
where $ \theta_k = h(x_k)/g(x_k) $. This iterative scheme is generally impractical, as solving problem \cref{Eq:n1} in each iteration is exceedingly challenging and expensive.  
Motivated by the quadratic transform \cite{Benson2004on, boct2023inertial}, 
which introduces an auxiliary variable and provides an equivalent non-fractional 
representation of a fractional term, Zhou et al.\ \cite{zhou2025min} 
proposed an equivalent non-fractional min--max reformulation of 
problem~\cref{Eq:n1}, i.e., $\min_{\substack{x \in \Omega \cap \mathcal{C}}} \max_{c \in \mathbb{R}} 
\left\{ 2c h(x) - c^2 h(x) g(x) + f(x) \right\}$,
where $ \mathcal{C} \subseteq \mathbb{R}^n $ is closed and $ \mathcal{C} \cap \Omega \neq \emptyset $ (with $ \Omega := \{x \in \mathbb{R}^n: g(x) \neq 0\} $). 
To address the nonsmooth and nonconvex–concave min–max problem, they further proposed an alternating maximization proximal descent algorithm (AMPDA).

When the fractional structure is absent, the problem reduces to a structured optimization formulation with nonlinear coupling constraints, for which the 
alternating direction method of multipliers (ADMM) is a popular solution method. Further discussion 
on ADMM can be found in the surveys \cite{boyd2011distributed, Han2022JORSC}.
Bolte et al.\ \cite{bolte2018nonconvex} proposed the information region within 
the augmented Lagrangian (AL) framework for problems with nonlinear coupling 
constraints; however, obtaining the uniform regularity parameter remains a 
nontrivial task.  
Cohen et al. \cite{cohen2022dynamic} developed a proximal linearized ADMM incorporating a backtracking procedure, in which function values are repeatedly evaluated to ensure the descent property. Hallak and Teboulle \cite{hallak2023adaptive} proposed an adaptive 
Lagrangian minimization scheme involving an update mechanism in which feasibility-penalizing terms are adjusted only when the feasibility gap exceeds predefined thresholds, with convergence to an $\varepsilon$-critical point. 
Zhu et al. \cite{zhu2024first} developed a flexible first-order primal–dual method via linearization and Bregman distances to approximate primal subproblems.

\subsection{Motivation}

Problem~\cref{Equ:1} is a nonconvex composite-structured fractional program with nonlinear coupling constraints. To the best of our knowledge, as noted earlier, such problems frequently arise in engineering applications, yet the development of effective solution methods remains relatively limited. 
If the fractional term in the objective function of problem~\cref{Equ:1} is 
treated as a standalone function, some structured methods are available for 
solving it; however, they impose more stringent assumptions. 
Meanwhile, in the existing nonconvex splitting methods, we also observe two common issues.  
(i) For ADMM-type splitting methods, the multiplier update typically adopts a 
fixed step size $ r = 1 $, e.g., 
\cite{ames2016alternating,banert2021fixing,boct2019admm,bourkhissi2025convergence}, 
which limits algorithmic flexibility. Subsequent studies 
\cite{2025Bai,boct2020proximal,yashtini2022convergence} extended this range to 
$(0,2)$ for nonconvex problems without fractional structures under linear 
constraints, but the composite-structured fractional optimization problem with 
nonlinear constraints still requires further investigation. 
(ii) The Kurdyka--{\L}ojasiewicz (KL) property \cite{2010Attouch} is a well-known and effective tool for analyzing the global convergence and convergence 
rate of nonconvex optimization methods, e.g., 
\cite{bolte2018nonconvex,boct2019admm,cohen2022dynamic,zhou2025min}, as it 
characterizes an inequality relationship between the function residual and the 
distance to its subdifferential. Zhu et al.\ \cite{zhu2024first} introduced the H\"olderian value proximity error bound (HVP-EB) condition, which establishes a similar inequality between the function residual and the norm of the variable residual for nonlinear constrained composite optimization problems without fractional structure. Notably, Zhu et al.\ \cite{zhu2024first} only established a linear 
convergence rate for exponents $ \theta \in (0,\frac{1}{2}) $. Moreover, the KL property implies the HVP-EB condition, meaning that the KL property can be viewed as a special case of the HVP-EB condition. This also indicates that the HVP-EB condition is a natural setting and is therefore more general than the KL property.

\subsection{Contributions} 

Motivated by the aforementioned considerations, we propose a Bregman proximal linearized ADMM (BPLADMM) with a relaxed multiplier update to provide a flexible and theoretically reliable iterative scheme for solving problem~\cref{Equ:1}. In summary, the main contributions of this work are twofold.
\begin{enumerate} 
\item 	For the proposed BPLADMM designed for problem~\cref{Equ:1}, we follow the idea in \cite{zhou2025min} and eliminate the fractional term in the objective function by introducing a new variable. This differs from some existing methods 
for fractional programming, such as the Dinkelbach method 
\cite{crouzeix1985algorithm} and its variants 
\cite{boct2017proximal,li2022proximal}.  
In addition, we introduce two notions of criticality following Bo{\c{t}} et al. \cite{boct2023inertial}, where 
their method requires the entire fractional term to satisfy certain assumptions in order to establish global convergence that are not required in our proposed BPLADMM.  Moreover, we incorporate the multiplier step size $ r \in (0,2) $ into the BPLADMM framework for solving problem~\cref{Equ:1}, making this the first work to achieve such a setting in composite-structured fractional programming with nonlinear coupling constraints.

\item Under standard assumptions, we establish that the sequence generated by BPLADMM converges subsequentially to a lifted critical point of the problem. Moreover, under mild conditions—including the use of the HVP-EB condition with exponent $\theta \in [0,1)$, which is more general than the KL property—we establish the 
global convergence of the entire sequence generated by BPLADMM and derive convergence rates for both the Lyapunov function sequence and the iterative sequence that improve upon those in \cite{zhu2024first}. Specifically, if $\theta = 0$, both the Lyapunov function and the iterates converge in a finite number of iterations; if $\theta \in (0, \tfrac{1}{2})$, 
the convergence is superlinear; if $\theta = \tfrac{1}{2}$, the convergence is linear; and if $\theta \in (\tfrac{1}{2}, 1)$, the convergence is sublinear.
\end{enumerate}

\subsection{Organization} 

The remainder of this paper is organized as follows. \Cref{sec:preliminaries} introduces basic notations and the necessary preliminary results. \Cref{sec:bpladmm} elaborates the iterative steps and theoretical results of the proposed method. Numerical evaluations are provided in \Cref{sec:numerical}, and conclusions are drawn in \Cref{sec:conclusions}.

\section{Preliminaries}\label{sec:preliminaries}

In this section, we present the foundational notations and preliminary results that serve as the theoretical basis for the next section.

The symbol $\|\cdot\|$ denotes the standard Euclidean norm for both vectors 
and matrices. The standard inner product associated with the $n$-dimensional 
Euclidean space $\mathbb{R}^n$ is defined as $\langle x, y \rangle := x^\top y$ 
for any vectors $x, y \in \mathbb{R}^n$. Given a nonempty subset $\mathcal{S}$ of 
$\mathbb{R}^n$ and a point $x \in \mathbb{R}^n$, the distance from $x$ to 
the set $\mathcal{S}$ is defined as $\dist(x, \mathcal{S}) := \inf\{\|y - x\|: y \in \mathcal{S}\}$, 
with $\dist(x, \emptyset) := +\infty$. 

For a function $f: \mathbb{R}^n \to (-\infty, +\infty]$, the domain is 
$\dom f := \{x \in \mathbb{R}^n: f(x) < +\infty\}$. A function $f$ 
is proper if $\dom f \neq \emptyset$ and $f(x) > -\infty$ for all 
$x \in \mathbb{R}^n$.   
The function $f$ is l.s.c. at $x \in \mathbb{R}^n$ 
if $f(x) \leq \liminf_{z \to x} f(z)$, and $f$ is called l.s.c. if this 
condition holds for all $x\in \mathbb{R}^n$. For a differentiable function 
$f: \mathbb{R}^n \to \mathbb{R}$, $\nabla f(x) \in \mathbb{R}^n$ denotes its 
gradient at $x$. For a differentiable mapping $F: \mathbb{R}^n \to \mathbb{R}^q$, 
$\nabla F(x) \in \mathbb{R}^{q \times n}$ denotes its Jacobian matrix at $x$. 
For a proper l.s.c. function $f: \mathbb{R}^n \to (-\infty, +\infty]$, the 
Fr\'echet subdifferential \cite[Definition 1.1]{Kruger2003on} at $x \in \dom f$ 
is defined as  
\begin{align*}
\widehat{\partial}f(x) := \left\{ d \in \mathbb{R}^n: \liminf_{\substack{y \to x \\ y \neq x}} 
\frac{f(y) - f(x) - \langle d, y - x \rangle}{\|y - x\|} \geq 0 \right\},
\end{align*}  
while the limiting subdifferential \cite[Definition 8.3]{Rockafellar2004var} at $x \in \dom f$ 
is defined as $\partial f(x) := 
\left\{ d\in \mathbb{R}^n: \exists \,x_k \xrightarrow{f} x,\ d_k \to d \text{ with } d_k \in \widehat\partial {f}(x_k)  \right\},$
and the horizon subdifferential \cite[Definition 8.3]{Rockafellar2004var} of $f$ at $x \in \dom f$ 
is defined as  
\begin{align*}
\partial^\infty f(x) := \left\{ d\in \mathbb{R}^n: \exists\, x_k \xrightarrow{f} x,\ 
\lambda_k \downarrow 0,\ \lambda_k d_k \to d \ \text{with}\ d_k \in \widehat{\partial}f(x_k) \right\},
\end{align*}  
where the notation $x_k \xrightarrow{f} x$ means $x_k \to x$ with $f(x_k) \to f(x)$. 
Both subdifferentials are empty for $x \notin \dom f$, 
i.e., $\widehat{\partial}f(x) = \partial f(x) = \emptyset$. 
If the function $f$ is finite at $x$, the subdifferentials $\widehat{\partial}f(x)$ and ${\partial}f(x)$ satisfy $\widehat{\partial}f(x)\subseteq {\partial}f(x)$ \cite[Theorem 8.6]{Rockafellar2004var}.
In the convex case, these 
subdifferentials reduce to the classical convex subdifferential:  $\widehat{\partial}f(x) = \partial f(x) = \left\{ d \in \mathbb{R}^n: f(y) \geq f(x) + \langle d, y - x \rangle,  \text{for all} \;y \in \mathbb{R}^n \right\}$ \cite[Proposition 8.12]{Rockafellar2004var}.  
For functions $f$ strictly differentiable at $x$, the Fr\'echet and limiting subdifferentials collapse 
to $\widehat{\partial}f(x) = \partial f(x) = \{\nabla f(x)\}$ \cite[Theorem 9.18]{Rockafellar2004var}. 
A point $x$ is termed critical if $0 \in \partial f(x)$ \cite[Theorem 10.1]{Rockafellar2004var}. 
We say that $f$ is regular at $x$ iff $\widehat{\partial} f(x) = \partial f(x)$. The function $f$ is said to be regular on a set $\mathcal{S}$ if it is regular at every point $x \in \mathcal{S}$.
A nonempty set $\mathcal{S}\subseteq \mathbb{R}^n$ is said to be regular at a point $x\in \mathcal{S}$ if the indicator function $\delta_{\mathcal{S}}$ is regular at $x$. The set $\mathcal{S}$ is called regular if it is regular at every point in $\mathcal{S}$.  

Let $\varphi: \mathbb{R}^n \to (-\infty, +\infty]$ be a proper function. Its conjugate function $\varphi^*: \mathbb{R}^n \to (-\infty, +\infty]$ is defined for $y \in \mathbb{R}^n$ by
$\varphi^*(y) := \sup\{\langle x, y \rangle - \varphi(x): x \in \mathbb{R}^n\}$. The function $\varphi^*$ is l.s.c. and convex. Furthermore, if $\varphi$ is l.s.c. and convex, then $\varphi^*$ is proper and, by \cite[Proposition 11.3]{Rockafellar2004var}, for any $x, y \in \mathbb{R}^n$,   
$\langle x, y \rangle = \varphi(x) + \varphi^*(y) \iff y \in \partial \varphi(x) \iff x \in \partial \varphi^*(y).$ 
Consequently, the Fenchel-Young inequality
$\varphi(x) + \varphi^*(y) \geq \langle x, y \rangle$ 
holds for all $x, y \in \mathbb{R}^n$, with equality iff $y \in \partial \varphi(x)$. 
Given a continuously differentiable convex function $\phi: \mathbb{R}^n \to \mathbb{R}$, the Bregman distance \cite{Bregman1967} associated with $\phi$ is defined as $D_{\phi}(x_1, x_2) = \phi(x_1) - \phi(x_2) - \langle \nabla \phi(x_2), x_1 - x_2 \rangle$, for all $x_1, x_2 \in \mathbb{R}^n$.

\begin{lemma}[{\cite{Wang2014arXiv}}] \label{1122lem:1}
Let $\phi: \mathbb{R}^n \to \mathbb{R}$ be a differentiable convex function, and 
let $D_{\phi}$ denote the Bregman distance associated with $\phi$. 
\begin{enumerate}
\item For all $x_1, x_2 \in \mathbb{R}^n$, the Bregman distance satisfies 
$D_{\phi}(x_1, x_2) \geq 0$; moreover, if $\phi$ is strictly convex, then $D_{\phi}(x_1, x_2) = 0$ iff $x_1 = x_2$.
\item For each fixed $x \in \mathbb{R}^n$, the function 
$D_{\phi}(\cdot, x)$ is convex.
\item If $\phi$ is $\alpha$-strongly convex, then for all 
$x_1, x_2 \in \mathbb{R}^n$, $D_{\phi}(x_1, x_2) \geq \frac{\alpha}{2}\|x_1 - x_2\|^2$. 
\item If $\nabla \phi$ is $L_{\nabla\phi}$-Lipschitz continuous, then for all 
$x_1, x_2 \in \mathbb{R}^n$,
$D_{\phi}(x_1, x_2) \leq \frac{L_{\nabla\phi}}{2}\|x_1 - x_2\|^2$. 
\end{enumerate} 
\end{lemma}

Now, we introduce some limiting subdifferential calculus rules \cite{boct2023inertial}. 

\begin{lemma}\label{lem:new1}
Let $\varphi_1, \varphi_2: \mathbb{R}^n \to (-\infty, +\infty]$ be proper l.s.c. functions, 
and let $x^* \in \mathbb{R}^n$.  
\begin{enumerate}

\item For any $k_1 \geq 0$, one has $\partial (k_1 \varphi_1)(x^*) = k_1 \partial \varphi_1(x^*)$. 
For any $k_2 < 0$, if $\varphi_1$ is Lipschitz continuous around $x^*$ and 
$\widehat{\partial}\varphi_1$ is nonempty-valued in a neighborhood of $x^*$, then 
$\partial (k_2 \varphi_1)(x^*) \subseteq k_2 \partial \varphi_1(x^*)$. 

\item Suppose that $\varphi_1$ is finite-valued at $x^*$ and $\varphi_2$ is locally 
Lipschitz around $x^*$. Then
$\partial(\varphi_1 + \varphi_2)(x^*) 
\subseteq \partial \varphi_1(x^*) + \partial \varphi_2(x^*)$, 
with equality if both $\varphi_1$ and $\varphi_2$ are regular at $x^*$; in this case, 
$\varphi_1 + \varphi_2$ is also regular at $x^*$. Moreover, if $\varphi_2$ is 
continuously differentiable at $x^*$, then 
$\partial(\varphi_1 + \varphi_2)(x^*) 
= \partial \varphi_1(x^*) + \nabla \varphi_2(x^*)$. 

\item Assume that $\varphi_1$ and $\varphi_2$ are Lipschitz continuous around $x^*$, 
and that $\widehat{\partial}\varphi_2$ is nonempty-valued in a neighborhood of $x^*$. 
Then
\begin{align*}
\partial\!\left(\frac{\varphi_1}{\varphi_2}\right)(x^*) 
\subseteq 
\frac{\partial\!\left(\varphi_2(x^*) \varphi_1\right)(x^*) 
- \varphi_1(x^*)\, \partial \varphi_2(x^*)}
{[\varphi_2(x^*)]^2}.
\end{align*}
If $\varphi_2$ is continuously differentiable at $x^*$, then
\begin{align*}
\partial\!\left(\frac{\varphi_1}{\varphi_2}\right)(x^*) 
=
\frac{\partial\!\left(\varphi_2(x^*) \varphi_1\right)(x^*) 
- \varphi_1(x^*)\, \nabla \varphi_2(x^*)}
{[\varphi_2(x^*)]^2},
\end{align*}
and consequently, $\varphi_1/\varphi_2$ is regular at $x^*$ iff the function $\varphi_2(x^*)\varphi_1(x)$ is regular at $x^*$.
\end{enumerate} 
\end{lemma}

Now, following \cite [Chapter~10]{Rockafellar2004var}, we introduce the first-order optimality conditions for problem~\cref{Equ:1}.

\begin{definition}[First-order optimality conditions]\label{lemma:1} 
For problem~\eqref{Equ:1}, we say that $(x^*,u^*,\lambda^*) \in \mathbb{R}^n\times \mathbb{R}^m\times \mathbb{R}^q$ is
\begin{enumerate}
\item a \emph{critical point} if $\nabla f(u^*) = M^\top \lambda^*$, $ F(x^*) = M u^*$, and $-\nabla F(x^*)^\top \lambda^* \in \partial \left({h }/{g}+\delta_\mathcal{S}\right)(x^*)$;

\item a \emph{lifted critical point} if 
\begin{align}\label{eq:3}
\left\{
\begin{aligned}
0 &\in \partial\delta_\mathcal{S}(x^*) 
+ \frac{\partial \left(g(x^*) h\right)(x^*) - h(x^*)\partial g(x^*)}{[g(x^*)]^2} 
+ \nabla F(x^*)^\top \lambda^*, \\
0 &= \nabla f(u^*) - M^\top \lambda^*,\,
0 = F(x^*) - M u^*.
\end{aligned}
\right.
\end{align}
\end{enumerate}
\end{definition}
Any local minimizer $(x^*,u^*)$ of problem \eqref{Equ:1} admits $\lambda^*\in \mathbb{R}^q$ such that $(x^*,u^*,\lambda^*)$ is a critical point.
Conversely, let $(x^*,u^*,\lambda^*)$ be a critical point satisfying  
$x^*\in \mathcal{S}\cap\ri(\dom g)\cap\ri(\dom h)$.  
If $\delta_\mathcal{S}$ is regular at $x^*$, 
Lemma~\ref{lem:new1}~(ii) and (iii) imply that $(x^*,u^*,\lambda^*)$ is also a lifted critical point of problem~\eqref{Equ:1}.  
The critical points and lifted critical points are equivalent in either of two cases:  
(i) both $g$ and $h$ are continuously differentiable;  
(ii) $g$ is continuously differentiable, and $\delta_\mathcal{S}$ is regular at $x^*$.

\begin{lemma}[{\cite{Pham2024bregman}}]\label{0611lem:2}
Let $L$ be a symmetric matrix in $\mathbb{R}^{d\times d}$ and let $M$ be a matrix in $\mathbb{R}^{p\times d}$. For all $x \in \mathbb{R}^d$,
$
\gamma (x^\top L x) \leq \|Lx\|^2$,
where $\gamma = \min\{|\sigma |: \sigma $ is a nonzero eigenvalue of $L\}$. Consequently, for all $z \in \Image(M)$,
$
\sqrt{\sigma_*(M^\top M)} \|z\| \leq \|M^\top z\|$,
where $\sigma_*(M^\top M)$ is the smallest strictly positive eigenvalue of $M^\top M$.
\end{lemma}

\begin{lemma}[{\cite{zhou2025min}}]\label{lem:2}
Let $\varphi_l, \varphi_c, \varphi_r: \mathbb{R}^n \to (-\infty, +\infty]$ be functions, and define
$\Phi(x) = \varphi_l(x)\big(\varphi_c(x)+\varphi_r(x)\big)$
for all $x \in \dom(\varphi_l) \cap \dom(\varphi_c) \cap \dom(\varphi_r)$, and $\Phi(x) = +\infty$ otherwise.
Let $\bar{x} \in \dom(\varphi_l) \cap \dom(\varphi_c) \cap \dom(\varphi_r)$, and suppose that 
\begin{enumerate}
    \item $\varphi_l$ is nonnegative and locally Lipschitz continuous around $\bar{x}$;
    \item $\varphi_c$ is convex and $\partial \varphi_c(\bar{x}) \neq \emptyset$;
    \item $\varphi_r$ is continuous at $\bar{x}$.  
\end{enumerate}
It then holds that
\(\widehat\partial \Phi(\bar{x}) \supseteq \widehat\partial \Psi(\bar{x}) + \varphi_l(\bar{x})\partial \varphi_c(\bar{x}),\)
where the extended-real-valued mapping \(\Psi\colon\mathbb R^n\to(-\infty,+\infty]\) is defined by
\[
\Psi(x)
=\begin{cases}
\big(\varphi_c(\bar{x})+\varphi_r(\bar{x})\big)\varphi_l(x)+\varphi_l(\bar{x})\varphi_r(x),
& x\in\dom\varphi_l\cap\dom\varphi_r,\\
+\infty, & \text{otherwise}.
\end{cases}
\]
\end{lemma}

\begin{lemma}[{\cite{Bento2025Convergence}}] \label{lem:15}
Let $\sigma > 0$ and $p \in (0,1)$, and let $\{a_k\} \subset (0, \infty)$ be a 
nonincreasing sequence. Assume that $a_{k+1} \leq a_k - \sigma a_{k+1}^p$ 
for all $k \in \mathbb{N}$. If $\{a_k\}$ converges to $\bar{a}$, then $\bar{a} = 0$, 
and $\{a_k\}$ converges  Q-superlinearly to 0.
\end{lemma}

The KL property was originally introduced by Kurdyka \cite{kurdyka1998gradients} as a generalization of the {\L}ojasiewicz inequality \cite{Thom1965}, and was later extended to nonsmooth settings~\cite{2010Attouch,bolte2014proximal}.
\begin{definition}\label{0129def:1}
Let $s \in (0, +\infty]$. We denote by $\Phi_s$ the set of all concave and continuous functions $\varphi: [0, s) \to [0, +\infty)$ that satisfy the following conditions: (i) $\varphi(0) = 0$; (ii) $\varphi$ is $\mathcal{C}^1$ on $(0, s)$ and continuous at $0$; (iii) for all $t \in (0, s)$, $\varphi'(t) > 0$.
\end{definition}

\begin{definition}[KL property]\label{def:kl}
Let $f: \mathbb{R}^n \to (-\infty, +\infty]$ be proper l.s.c. and $\bar{x} \in \dom \partial f$. Denote 
$[\nu_1 < f < \nu_2] := \{x \in \mathbb{R}^n: \nu_1 < f(x) < \nu_2\}$. 
We say $f$ satisfies the KL property at $\bar{x}$ if there exist $\nu \in (0, +\infty]$, 
a neighborhood $U$ of $\bar{x}$, and a continuous concave function $\varphi \in \Phi_{\nu}$ 
such that: 
\[\varphi'(f(x) - f(\bar{x})) \cdot \dist(0, \partial f(x)) \geq 1,\, \forall x \in U \cap [f(\bar{x}) < f < f(\bar{x}) + \nu].\]
The function $\varphi$ is termed the desingularizing function. If $f$ satisfies the KL property 
at every point in $\dom \partial f$, it is called a KL function.
\end{definition}

\begin{lemma}[Uniformized KL property]\label{0624lem:1}
Let $\Omega \subseteq \mathbb{R}^n$ be compact and $f: \mathbb{R}^n \to (-\infty, +\infty]$ proper l.s.c. 
and constant on $\Omega$. If $f$ satisfies the KL property on $\Omega$, then 
there exist $\eta > 0$, $\nu > 0$, and $\varphi \in \Phi_{\nu}$ such that for all 
$\bar{x} \in \Omega$ and $x \in \{ x \in \mathbb{R}^n: \dist(x, \Omega) < \eta \} \cap 
[f(\bar{x}) < f(x) < f(\bar{x}) + \nu],$ 
the following inequality holds:
\begin{align}\label{0127eq:1}
\varphi'(f(x) - f(\bar{x})) \cdot \dist(0, \partial f(x)) \geq 1.  
\end{align}
\end{lemma}

\section{BPLADMM and convergence analysis}\label{sec:bpladmm}

In this section, we present the iterative steps of the proposed BPLADMM for solving problem~\cref{Equ:1} and analyze its convergence properties. For clarity and simplicity in the theoretical analysis, we use $\{z_k\}_{k\ge 1}$ to denote the sequence $\{x_k, u_k, \lambda_k\}_{k\ge 1}$. Throughout the analysis, the sequence $\{z_k\}_{k\ge 1}$ is assumed to be generated by the proposed BPLADMM. 

We first introduce the AL function associated with 
problem~\cref{Equ:1}, defined as 
\[\mathcal{L}_\rho(x, u, \lambda)
= f(u) +\delta_{\mathcal{S}}(x)+ \frac{h(x)}{g(x)}
+ \langle \lambda, F(x) - Mu \rangle
+ \frac{\rho}{2} \|F(x) - Mu\|^2,\]
where $\rho \in \mathbb{R}_{++}$ is a penalty parameter and 
$\lambda \in \mathbb{R}^q$ denotes the Lagrange multiplier. To address fractional term $h(x)/g(x)$, we follow the approach 
in \cite{zhou2025min} and eliminate the fractional structure by introducing a 
new variable, that is $h(x)/g(x)=\max_{c \in \mathbb{R}} \left\{ 2ch(x) - c^2h(x)g(x) \right\}$.
For a fixed $x^*\in \mathcal{S}$, the function of $c$ on the right-hand side is concave, with $c^* = 1 / g(x^*)$ as a maximizer, which is unique if $h(x^*) >0$. Furthermore, we provide a non-fractional reformulation of $\mathcal{L}_\rho$, namely, 
\begin{equation}\label{1008eq:1}
\mathscr{L}_\rho(x,u,\lambda,c,y)
= f(u)+\delta_{\mathcal{S}}(x) + 2ch(x) + c^2h(x)\big(g^*(y) - \langle x, y \rangle\big) + \psi(x,u,\lambda),
\end{equation}
where $g^*(y)$ denotes the conjugate function of $g(x)$, and 
$\psi(x,u,\lambda)
:= \langle \lambda, F(x) - Mu \rangle
+ \frac{\rho}{2}\|F(x) - Mu\|^2$. 
It is well known that $\mathcal{L}_\rho(x,u,\lambda)
\leq \mathscr{L}_\rho\big(x,u,\lambda,1/g(x),y\big)$, 
with equality iff $y \in \partial g(x)$.
The following lemma establishes the relationship between the critical points of problems \cref{Equ:1} and \cref{1008eq:1}.

\begin{lemma}\label{lemma:2}
Suppose that $0 \in \partial \mathscr{L}_\rho(x^*,u^*,\lambda^*,c^*,y^*)$.
Let $c^* = 1/g(x^*)$, $y^* \in \partial g(x^*)$, and assume that $x^* \in \mathcal{S}
\cap \ri(\dom h)$ and $\delta_\mathcal{S}$ is regular on $\mathcal{S}$.
Then $(x^*,u^*,\lambda^*)$ is a lifted critical point of problem~\cref{Equ:1}.
\end{lemma}
\begin{proof}
Let $\eta(x) := c^*  h(x) + (c^*)^2\big[ h(x)- h(x^*)\big]
\big(g^*(y^*) - \langle x, y^* \rangle\big)$.
Then
\begin{align}\label{1030eq2}
\lim_{x,z \to x^*,\ x\neq z} \frac{|\eta(x) -\eta(z)|}{\|x -z\|}
\leq (c^*)^2 L_h\|y^*\|\lim_{x,z \to x^*,\ x\neq z}\left(\|x -x^*\| +\|z -x^*\|\right) =0,
\end{align}
where the inequality follows from the Lipschitz continuity of $ h$
around $x^*$.
Consequently, $\eta$ is strictly differentiable at $x^*$ with $\nabla \eta(x^*) = 0$.
For $c^* = 1/g(x^*)$ and $y^* \in \partial g(x^*)$, we obtain
\begin{align*}
\partial_x \mathscr{L}_\rho
&= \partial\!\left(\delta_\mathcal{S}+
2c^* h + (c^*)^2  h\big(g^*(y^*) - \langle x, y^* \rangle\big)
+ \psi(x,u^*,\lambda^*)
\right)\!(x^*) \\
&= \partial\!\left(\delta_\mathcal{S}+
c^* h + (c^*)^2  h(x^*)\big(g^*(y^*) - \langle x, y^* \rangle\big)
\right)\!(x^*)
+ \nabla \eta(x^*) \\
&\quad + \nabla F(x^*)^\top
\big(\lambda^* + \rho(F(x^*) - M u^*)\big) \\
&=\partial\delta_\mathcal{S}(x^*)+ c^* \partial  h(x^*)
- (c^*)^2  h(x^*) y^*
+ \nabla F(x^*)^\top
\big(\lambda^* + \rho(F(x^*) - M u^*)\big),
\end{align*}
where the last equality follows from the regularity of $\delta_\mathcal{S}$ and $c^*h$ at $x^*$.
Furthermore,  $\partial_u \mathscr{L}_\rho 
= \nabla f(u^*) - M^\top \big(\lambda^* + \rho(F(x^*) - M u^*)\big)$ and $\partial_\lambda \mathscr{L}_\rho = F(x^*) - M u^*$.  
Thus, if  
$0 \in \partial \mathscr{L}_\rho(x^*,u^*,\lambda^*,c^*,$ $y^*)$,  
then \cref{eq:3} holds. The proof is complete. 
\end{proof}

Before presenting the algorithm, we define the following linearized notations for the nonlinear terms in the AL function: $\ell_{F}(x;\bar{x}):=F(\bar{x})+\nabla F(\bar{x})(x-\bar{x})$ and $\widetilde \psi(x,u,\lambda;\bar{x}):=\langle \lambda,\ell_F(x;\bar{x})-Mu\rangle+\frac{\rho}{2}\|\ell_F(x;\bar{x})-Mu\|^2$.

\subsection{The proposed BPLADMM}

In this section, we describe the iterative steps of the proposed BPLADMM in Algorithm~\ref{alg:1}. It is worth noting that, for the \(x\)-subproblem, we linearize the composite term \(c h(x) + c^{2} h(x)(g^\ast(y) - \langle x, y \rangle)\) in \(\mathscr{L}_\rho\) with respect to \(x\). Specifically, define 
$\eta := c_k h(x) + c_k^{2} (h(x) - h(x_k))(g^\ast(y_k) - \langle x, y_k \rangle)$. 
Similar to~\cref{1030eq2}, it follows that \(\nabla \eta(x_k) = 0\). Consequently, we obtain 
$\nabla_x\!\left[\, c_k h(x_k) + c_k^{2} h(x_k)(g^\ast(y_k) - \langle x_k, y_k \rangle) \right] = -c_k^{2} h(x_k) y_k$. Furthermore, we replace $\psi(x,u,\lambda)$ with $\widetilde\psi(x,u,\lambda;\bar{x})$, where the latter is obtained by linearizing the term involving the nonlinear mapping $F$ in \cref{1008eq:1} to reduce the computational cost.
Furthermore, we incorporate Bregman distances into the $x$-subproblem, which leads to easy-to-solve updates and even closed-form solutions. 

\begin{algorithm}[!htb]
\caption{BPLADMM}\label{alg:1}
\begin{algorithmic}[1]
\REQUIRE  Let $0 < r < 2$, $\rho > 0$, and let $\{\phi_k\}$ be a sequence of differentiable functions, uniformly $\tau$-strongly convex and with uniformly $L_{\nabla\phi}$-Lipschitz continuous gradients.
\STATE \textbf{Initialize:} $(x_0, u_0, \lambda_0)$ with $x_0\in \mathcal{S}$, $\lambda_0 \in \Image(M)$ and $c_0 = 1/g(x_0)$. 
\WHILE{the stopping criterion is not met}
\STATE \textbf{Step 1: $x$-update} 
\STATE Choose $ y_k\in \partial g(x_k) $.
\begin{align}\label{1129Equ3.5}
x_{k+1}=\underset{x\in \mathcal{S}}{\arg \min} \left\{c_k h(x)-\langle c_k^2h(x_k)y_k, x-x_k \rangle \!+\!\widetilde\psi(x,u_k,\lambda_k;x_k)\!+\!D_{\phi_k}(x,x_k)\right\}.
\end{align}
\vspace{-1.5em}
\STATE Compute $ c_{k+1}=1/g(x_{k+1}) $.
\STATE \textbf{Step 2: $u$-update}
\begin{align}\label{Equ:10}
u_{k+1} =\underset{u\in \mathbb{R}^m}{\arg \min}\, 
\mathscr{L}_\rho(x_{k+1},u,\lambda_{k},c_{k+1},y_k).  
\end{align}	
\vspace{-1.5em}
\STATE \textbf{Step 3: $\lambda$-update with over-relaxation}
\begin{align}
\lambda_{k+1} = \lambda_k + r\rho (F(x_{k+1}) - Mu_{k+1}). \label{Equ:12}
\end{align}	
\vspace{-1.5em}
\STATE Update $k \leftarrow k + 1$ and go to Step 1.
\ENDWHILE
\ENSURE $(x_{k+1}, u_{k+1}, \lambda_{k+1})$.
\end{algorithmic}
\end{algorithm}

For the convergence analysis, we impose the following standing assumptions.    
\begin{assumption}\label{ass:1}
~
\begin{enumerate}
\item The set $\mathcal{S}$ is nonempty, closed, and regular.
\item The gradient $\nabla f$ is $L_{\nabla f}$-Lipschitz continuous.
\item The Jacobian $\nabla F$ is $L_{\nabla F}$-Lipschitz continuous.
\item The function $h$ is a proper l.s.c. convex function that is nonnegative and Lipschitz continuous on an open neighborhood of $\mathcal{S}$ with constant $L_h$.
\item The function $g$ is a proper l.s.c. convex function that is positive on $\mathcal{S}$ and Lipschitz continuous on an open neighborhood of $\mathcal{S}$ with constant $L_g$.
\item The matrix $M \in \mathbb{R}^{q \times m}$ is full column rank, which implies $\sigma_{\min}(M^\top M)>0$, where $\sigma_{\min} (\cdot)$ denotes the smallest eigenvalue of a matrix.
\item The image of $F$ is contained within the image of $M$, i.e., $F(\mathcal{S}) \subseteq \Image(M)$.
\end{enumerate}
\end{assumption}

Before proceeding, we state an assumption on the iterate sequence generated by Algorithm~\ref{alg:1}, which plays a critical role in the subsequent convergence analysis.
\begin{assumption}\label{0902ass:1}
The sequence $\{z_k\}_{k\ge1}$ generated by Algorithm~\ref{alg:1} is bounded.
\end{assumption}

To clarify the notation for different regions, let $\mathcal{X}\subseteq \mathcal{S}\times \mathbb{R}^m\times\mathbb{R}^q$ denote the bounded region appearing in Assumption~\ref{0902ass:1}, and let $\bar{\mathcal{X}}$ stand for the closure of $\mathcal{X}$, i.e., $\operatorname{cl}(\mathcal{X})$. Let $\bar{\mathcal{X}}^x$, $\bar{\mathcal{X}}^u$, and $\bar{\mathcal{X}}^\lambda$ denote the projections of $\bar{\mathcal{X}}$ onto $\mathbb{R}^n$, $\mathbb{R}^m$, and $\mathbb{R}^q$, respectively.

\begin{lemma}\label{0902lem:1} 
Suppose that Assumptions~\ref{ass:1} and \ref{0902ass:1} hold, and that $g(x)$ is bounded below by $m_g>0$ on $\bar{\mathcal{X}}^x$. Moreover, there exists $\Delta>0$ such that the following hold.
Let 
$\bar{\mathcal{X}}_\Delta^x:=\big\{x\in \mathbb{R}^n:\dist(x,\bar{\mathcal{X}}^x)^2\leq \Delta\big\}$, $\bar{\mathcal{X}}_\Delta^u:=\big\{u\in \mathbb{R}^m:\dist(u,\bar{\mathcal{X}}^u)^2\leq \Delta\big\}$, and $\bar{\mathcal{X}}_\Delta^\lambda:=\big\{\lambda\in \mathbb{R}^q:\dist(\lambda,\bar{\mathcal{X}}^\lambda)^2\leq \Delta\big\}$. Define $\bar{\mathcal{X}}_\Delta=\bar{\mathcal{X}}_\Delta^x\times \bar{\mathcal{X}}_\Delta^u\times\bar{\mathcal{X}}_\Delta^\lambda$.
\begin{enumerate}
\item For any $x\in \bar{\mathcal{X}}_\Delta^x$, it holds that $g(x)\geq m_g/2$.
\item The constants below are finite: $M_h := \sup_{x \in \bar{\mathcal{X}}_\Delta^x} h(x)$, $M_g := \sup_{x \in \bar{\mathcal{X}}_\Delta^x} g(x)$, $M_{\nabla F} := \sup_{x \in \bar{\mathcal{X}}_\Delta^x} \|\nabla F(x)\|$ and 
$M_{h/g^2} := \sup_{x \in \bar{\mathcal{X}}_\Delta^x} {h(x)}/{g(x)^2}$. 
\item The functions $1/g$ and $h/g^2$ are Lipschitz continuous on $\bar{\mathcal{X}}_\Delta^x$ with Lipschitz constants $L_{1/g}$ and $L_{h/g^2}$, respectively.
\end{enumerate}
\end{lemma}
\begin{proof}
Since the continuous function $g$ is strictly positive on the compact set $\bar{\mathcal{X}}^x$, statement~(i) holds for a sufficiently small $\Delta>0$. By the compactness of $\bar{\mathcal{X}}_\Delta^x$, the continuity of $h$ and $g$, together with statement~(i), statement~(ii) holds. Using the definition of Lipschitz continuity combined with statements~(i) and (ii), we obtain statement~(iii).
\end{proof}

\begin{lemma}[{\cite[Lemma 4.1]{cohen2022dynamic}}]\label{0902lem:2}
Suppose that Assumptions~\ref{ass:1} and \ref{0902ass:1} hold. The function $\psi$ has a Lipschitz gradient on the compact set $\bar{\mathcal{X}}_\Delta$, with Lipschitz constant
\begin{align*}
L_{\psi}=\sup_{(x,u,\lambda)\in \bar{\mathcal{X}}_\Delta}\Big\{L_{\nabla F}\|\lambda+\rho(F(x)-Mu)\|\Big\}+(M_{\nabla F}+\|M\|)\big(2+\rho(M_{\nabla F}+\|M\|)\big).
\end{align*}
\end{lemma}

\subsection{Preliminary convergence results}\label{Sect3.2} 

This subsection presents the preliminary results of the proposed BPLADMM, 
including the descent property of the Lyapunov sequence and the convergence results of the iterative subsequence.

\begin{lemma}[Descent property of $\mathscr{L}_\rho$ with respect to the variables $x$, $u$, $c$ and $y$]\label{lem:3} 
Suppose that Assumptions~\ref{ass:1} and \ref{0902ass:1} hold. Then the reformulated AL function $\mathscr{L}_{\rho}$ satisfies
\begin{align*}
&\mathscr{L}_\rho(x_{k+1}, u_{k+1}, \lambda_k,c_{k+1},y_k) - \mathscr{L}_\rho(x_{k}, u_k, \lambda_k,c_k,y_{k-1}) \\
&\leq -\frac{\tau-2L_{h/{g^2}}L_g-L_\psi}{2} \|x_{k+1}-x_k\|^2-\frac{\rho\sigma_{\min}(M^\top M)-L_{\nabla f}}{2}\|u_{k+1}-u_k\|^2. 
\end{align*}
\end{lemma}

\begin{proof}
Using the definitions of $\mathscr{L}_\rho$, it directly follows that 
\begin{align*}
&\mathscr{L}_\rho(x_{k+1}, u_k, \lambda_k, c_{k+1}, y_k) + c_k h(x_k) + \psi(x_k, u_k, \lambda_k)- \mathscr{L}_\rho(x_k, u_k, \lambda_k,c_k,y_k)\\
&= 2c_{k+1} h(x_{k+1}) + c_{k+1}^2 h(x_{k+1}) (g^*(y_k) - \langle x_{k+1}, y_k \rangle) + \psi(x_{k+1}, u_k, \lambda_k) . 
\end{align*}
 For the $x$-subproblem~\eqref{1129Equ3.5}, we have 
 \[c_k h(x_{k+1}) + \psi(x_{k+1}, u_k, \lambda_k) -\langle c_k^2 h(x_k) y_k, x_{k+1} - x_k \rangle 
        \leq    c_k h(x_k) + \psi(x_k, u_k, \lambda_k)-D_{\phi}(x_{k+1}, x_k),\]
from which we deduce that 
\begin{align*}
& \mathscr{L}_\rho(x_{k+1}, u_k, \lambda_k, c_{k+1}, y_k) - \mathscr{L}_\rho(x_k, u_k, \lambda_k,c_k,y_k) \\
&\leq 2c_{k+1} h(x_{k+1}) + c_{k+1}^2 h(x_{k+1}) (g^*(y_k) - \langle x_{k+1}, y_k \rangle) - c_k h(x_{k+1})+\psi(x_{k+1},u_k,\lambda_k)\\
&\quad  + \langle c_k^2 h(x_k) y_k, x_{k+1} - x_k \rangle -  D_{\phi_k}(x_{k+1},x_k) -\widetilde\psi(x_{k+1},u_k,\lambda_k;x_k) \\
&= 2c_{k+1} h(x_{k+1}) + c_{k+1}^2 h(x_{k+1}) (g^*(y_k) - \langle x_{k+1}, y_k \rangle)  + \langle c_{k+1}^2 h(x_{k+1}) y_k, x_{k+1} - x_k \rangle \\
&\quad  + \langle c_k^2 h(x_k) y_k - c_{k+1}^2 h(x_{k+1}) y_k, x_{k+1} - x_k \rangle - c_k h(x_{k+1})- D_{\phi_k}(x_{k+1},x_k)\\
&\quad +\psi(x_{k+1},u_k,\lambda_k)-\widetilde\psi(x_{k+1},u_k,\lambda_k;x_k).
\end{align*}
This, together with Lemma~\ref{1122lem:1}~(iii), Lemma~\ref{0902lem:1}~(ii) and the bound $\|y\| \leq L_g$ (which follows from the Lipschitz continuity of $g(x)$ \cite[Corollary~1.81]{mordukhovich2006variational}), implies that 
\begin{align*}
& \mathscr{L}_\rho(x_{k+1}, u_k, \lambda_k, c_{k+1}, y_k) - \mathscr{L}_\rho(x_k, u_k, \lambda_k,c_k,y_k) \\
&\leq 2c_{k+1} h(x_{k+1}) + c_{k+1}^2 h(x_{k+1}) (g^*(y_k) - \langle x_{k+1}, y_k \rangle)  + \langle c_{k+1}^2 h(x_{k+1}) y_k, x_{k+1} - x_k \rangle \\
&\quad +\left( L_{{h}/{g^2}}  L_g - \frac{ \tau}{2} \right)  \| x_{k+1} - x_k \|^2\!-\! c_k h(x_{k+1})+\psi(x_{k+1},u_k,\lambda_k)-\widetilde\psi(x_{k+1},u_k,\lambda_k;x_k).
\end{align*}
We compute the right-hand side of the inequality in blocks. First,
\begin{align*}
&\psi(x_k,u_k,\lambda_k)-\widetilde\psi(x_{k+1},u_k,\lambda_k;x_k)
+\psi(x_{k+1},u_k,\lambda_k)-\psi(x_k,u_k,\lambda_k)\\
&= \left\langle \lambda_k,\,-\nabla F(x_k)(x_{k+1}-x_k)\right\rangle
+\frac{\rho}{2}\| F(x_k)-Mu_k\|^2+\psi(x_{k+1},u_k,\lambda_k)-\psi(x_k,u_k,\lambda_k)\\
&\quad -\frac{\rho}{2}\big\| F(x_k)+\nabla F(x_k)(x_{k+1}-x_k)-Mu_k \big\|^2\\
&= -\big\langle \nabla F(x_k)^\top \big(\lambda_k+\rho \big(F(x_k)-Mu_k\big)\big),\,x_{k+1}-x_k\big\rangle
-\frac{\rho}{2}\big\|\nabla F(x_k)(x_{k+1}-x_k)\big\|^2\\
&\quad +\psi(x_{k+1},u_k,\lambda_k)-\psi(x_k,u_k,\lambda_k)\\
&\leq \psi(x_{k+1},u_k,\lambda_k)-\psi(x_k,u_k,\lambda_k)
-\big\langle \nabla_x\psi(x_k,u_k,\lambda_k),\,x_{k+1}-x_k\big\rangle.
\end{align*}
By the Lipschitz continuity of $\nabla \psi$ from Lemma~\ref{0902lem:2}, we have \[  \psi(x_{k+1},u_k,\lambda_k)-\widetilde\psi(x_{k+1},u_k,\lambda_k;x_k)\leq \frac{L_\psi}{2}\|x_{k+1}-x_k\|^2.\]
Next,
\begin{align*}
&2c_{k+1} h(x_{k+1}) + c_{k+1}^2 h(x_{k+1}) (g^*(y_k) - \langle x_{k+1}, y_k \rangle) - c_k h(x_{k+1}) \\
&\quad + \langle c_{k+1}^2 h(x_{k+1}) y_k, x_{k+1} - x_k \rangle \\
&= c_k c_{k+1}^2 h(x_{k+1}) \left( 2g(x_k)g(x_{k+1}) - g(x_{k+1})^2 \right) + c_{k+1}^2 h(x_{k+1}) \left( g^*(y_k) - \langle x_k, y_k \rangle \right) \\
&= -c_k c_{k+1}^2 h(x_{k+1}) \left( g(x_{k+1}) - g(x_k) \right)^2 + c_{k+1}^2 h(x_{k+1}) \left( g(x_k) + g^*(y_k) - \langle x_k, y_k \rangle \right).
\end{align*}
Thus, the above quantity is nonpositive, implying that 
\[\mathscr{L}_\rho(x_{k+1}, u_k, \lambda_k,c_{k+1},y_k)+ \frac{\tau-2L_{{h}/{g^2}}L_g-L_\psi}{2}  \|x_{k+1}-x_k\|^2
\leq \mathscr{L}_\rho(x_k,u_k,\lambda_k,c_k,y_k).\]
Since \(\mathscr{L}_\rho(x_k,u_k,\lambda_k,c_k,y_k)\leq \mathscr{L}_\rho(x_k,u_k,\lambda_k,c_k,y_{k-1})\), we further infer that
\begin{align}\label{Equ:11}
&\mathscr{L}_\rho(x_{k+1},u_k,\lambda_k,c_{k+1},y_k)-\mathscr{L}_\rho(x_k,u_k,\lambda_k,c_k,y_{k-1}) \leq -\frac{\tau-2L_{{h}/{g^2}}L_g-L_\psi}{2} \|x_{k+1}-x_k\|^2.
\end{align}
Straightforward calculation gives
\begin{equation}\label{1008eq:2}
    \begin{aligned}
&\mathscr{L}_\rho(x_{k+1},u_{k+1},\lambda_k,c_{k+1},y_k)-\mathscr{L}_\rho(x_{k+1},u_k,\lambda_k,c_{k+1},y_{k}) \\
&= f(u_{k+1})-f(u_k)+\psi(x_{k+1},u_{k+1},\lambda_k)-\psi(x_{k+1},u_k,\lambda_k)\\
&=f(u_{k+1})-f(u_k)\!+\!\frac{\rho}{2}\left\langle  2F(x_{k+1})\!-\!M(u_{k+1}\!+\!u_k), M(u_k\!-\!u_{k+1})\right\rangle\!-\!\langle \lambda_k,M(u_{k+1}\!-\!u_k) \rangle\\
&=f(u_{k+1})\!-\!f(u_k)\!-\!\langle \lambda_k\!+\!\rho\left( F(x_{k+1})-Mu_{k+1}\right),M(u_{k+1}-u_k) \rangle\!-\!\frac{\rho}{2}\|M(u_{k+1}\!-\!u_k)\|^2.
\end{aligned}
\end{equation}
From the optimality condition of the $u$-subproblem, it follows that
\begin{equation}\label{1008eq:3}
    0=\nabla f(u_{k+1})-M^\top \left(\lambda_k+\rho\left(F(x_{k+1})-M u_{k+1}\right)\right).
\end{equation} 
Substituting \cref{1008eq:3} into \cref{1008eq:2}, we obtain
\begin{align*}
&\mathscr{L}_\rho(x_{k+1},u_{k+1},\lambda_k,c_{k+1},y_k)-\mathscr{L}_\rho(x_{k+1},u_k,\lambda_k,c_{k+1},y_{k}) \\
&=f(u_{k+1})-f(u_k)-\langle \nabla f(u_{k+1}),u_{k+1}-u_k \rangle-\frac{\rho}{2}\|M(u_{k+1}-u_k)\|^2.
\end{align*}
From Assumptions~\ref{ass:1}~(ii) and (vi) and the descent lemma in \cite{bertsekas1999nonlinear}, it follows that
\begin{align}\label{0429eq:1}
&\mathscr{L}_\rho(x_{k+1},u_{k+1},\lambda_k,c_{k+1},y_k)-\mathscr{L}_\rho(x_{k+1},u_k,\lambda_k,c_{k+1},y_{k}) \notag\\
&\leq -\frac{\rho\sigma_{\min}(M^\top M)-L_{\nabla f}}{2}\|u_{k+1}-u_k\|^2.
\end{align}
The combination of \cref{Equ:11} and \cref{0429eq:1} establishes the desired result. 
\end{proof}

\begin{lemma}[Descent property of $\mathscr{L}_\rho$ with respect to the multiplier $\lambda$]\label{lem:4}
Suppose that Assumptions~\ref{ass:1}~(ii),~(vi) and~(vii) hold. Then the reformulated AL function $\mathscr{L}_{\rho}$ satisfies
\begin{align*}
&\mathscr{L}_\rho(x_{k+1}, u_{k+1}, \lambda_{k+1},c_{k+1},y_k) - \mathscr{L}_\rho(x_{k+1}, u_{k+1}, \lambda_k,c_{k+1},y_k) \\
&\leq C_1\|u_{k+1}-u_k\|^2\!+C_2\left(\left\|M^\top(\lambda_{k}-\lambda_{k-1})\right\|^2-\left\|M^\top(\lambda_{k+1}-\lambda_{k})\right\|^2\right),
\end{align*}
where the constants $C_1$ and $C_2$ are defined as 
\begin{align*}
C_1 =\begin{cases}
\frac{L_{\nabla f}^2}{r \rho \sigma_{\min}\left(M^\top M\right)}, & 0<r\leq1,\\
\frac{ r L_{\nabla f}^2}{(2 - r)^2 \rho  \sigma_{\min}\left( M^\top M\right)},&1<r<2,
\end{cases} \ \ 
C_2 =\begin{cases}
\frac{1 - r}{r^2 \rho \sigma_{\min}\left(M^\top M\right)}, & 0<r\leq1,\\
\frac{r - 1}{(2 - r) r\rho  \sigma_{\min}\left(M^\top M\right)},&1<r<2.
\end{cases}
\end{align*}
\end{lemma}

\begin{proof}
By the definition of the reformulation AL function $\mathscr{L}_\rho$ and \cref{Equ:12}, this implies that
\[\mathscr{L}_\rho(x_{k+1}, u_{k+1}, \lambda_{k+1},c_{k+1},y_k) - \mathscr{L}_\rho(x_{k+1}, u_{k+1}, \lambda_k,c_{k+1}, y_{k})  = \frac{1}{r\rho}\|\lambda_{k+1} - \lambda_{k}\|^2.\]
Subsequently, the first-order optimality condition for the $u$-subproblem yields $
0 = \nabla f(u_{k+1}) - M^\top\lambda_k - \rho M^\top (F(x_{k+1}) - Mu_{k+1})$.
Combining this with the dual update rule \cref{Equ:12} leads to 
\[0 = \nabla f(u_{k+1}) - M^\top\lambda_k - \frac{M^\top (\lambda_{k+1} - \lambda_k)}{r},\] which can be rearranged to
$
M^\top\lambda_{k+1} = r\nabla f(u_{k+1}) + (1 - r)M^\top \lambda_k.  
$
By shifting the index $k \to k-1$, we obtain $M^\top\lambda_{k} = r\nabla f(u_{k}) + (1 - r) M^\top\lambda_{k-1}$. 
Subtracting the two expressions yields
\begin{align}\label{Equ:15}
M^\top(\lambda_{k+1} - \lambda_{k} )= r(\nabla f(u_{k+1}) - \nabla f(u_{k})) + (1 - r) M^\top(\lambda_k - \lambda_{k-1}).
\end{align}
Since $\lambda_0 \in \Image(M)$, $F(\mathcal{S}) \subseteq \Image(M)$ and by \eqref{Equ:12}, we deduce that $\lambda_{k+1} - \lambda_k \in \Image(M)$. 
We consider the following two cases separately.

(i) Given $0<r\leq1$, the convexity of the squared Euclidean norm $\|\cdot\|^2$ implies that 
\begin{equation}\label{1008eq:4}
    \left\|M^\top\left( \lambda_{k+1} - \lambda_k \right)\right\|^2 \leq r \left\| \nabla f(u_{k+1}) - \nabla f(u_{k}) \right\|^2 + (1 - r) \left\|M^\top\left( \lambda_k - \lambda_{k-1}\right) \right\|^2.
\end{equation}
Rearranging \cref{1008eq:4} and applying Assumption~\ref{ass:1}~(vi) together with Lemma~\ref{0611lem:2} yield
\begin{align*}
r\sigma_{\min} \left(M^\top M\right) \| \lambda_{k+1} - \lambda_k \|^2 &\leq (1 - r)\left( \left\|M^\top( \lambda_k - \lambda_{k-1}) \right\|^2 - \left\| M^\top (\lambda_{k+1} - \lambda_k )\right\|^2 \right)\\
&\quad +r\|\nabla f(u_{k+1}) - \nabla f(u_{k}) \|^2.
\end{align*}
Using the Lipschitz continuity of $\nabla f$  yields 
\begin{align*}
\frac{1}{r \rho} \| \lambda_{k+1} - \lambda_k \|^2 &\leq 
\frac{1 - r}{r^2 \rho \sigma_{\min}\left(M^\top M\right)} \left( \left\|M^\top (\lambda_k - \lambda_{k-1} )\right\|^2 - \left\| M^\top(\lambda_{k+1} - \lambda_k )\right\|^2 \right)\\
&\quad +\frac{ L_{\nabla f}^2 }{r \rho \sigma_{\min}\left(M^\top M\right)} \| u_{k+1} - u_k \|^2.
\end{align*}

(ii) Given $1<r<2$ and the recurrence relation \cref{Equ:15}, the dual increment can be rewritten as 
\begin{equation}\label{1008eq;5}
    M^\top\left(\lambda_{k+1} - \lambda_{k}\right) = (2 - r) \frac {r}{2 - r}\left (\nabla f(u_{k+1}) - \nabla f(u_{k}) \right) + (r - 1)M^\top (\lambda_{k-1} - \lambda_{k}).
\end{equation}
By the convexity of the squared Euclidean norm, \cref{1008eq;5} implies
\[\left\| M^\top(\lambda_{k+1}\! - \!\lambda_k )\right\|^2 \!\leq\! (2 \!-\! r) \left\| \frac{r}{2- r} (\nabla f(u_{k+1})\! -\! \nabla f(u_{k})) \right\|^2 \!+\! (r\! -\! 1) \left\|M^\top( \lambda_{k} \!-\! \lambda_{k-1}) \right\|^2.\] Since $0 < 2 - r < 1$, it follows that
\begin{equation}\label{1008eq:6}
    \left\| M^\top(\lambda_{k+1} - \lambda_k ) \right\|^2 \leq \frac{r^2}{2 - r} \| \nabla f(u_{k+1}) - \nabla f(u_{k})\|^2 + (r - 1) \left\| M^\top( \lambda_{k} - \lambda_{k-1})\right\|^2.
\end{equation}
Rearranging the \cref{1008eq:6} and applying Assumption~\ref{ass:1}~(vi) together with Lemma~\ref{0611lem:2} yield
\begin{multline*}
(2 - r) \sigma_{\min}\left(M^\top M\right)\| \lambda_{k+1} - \lambda_k \|^2 \leq \frac{r^2}{2 - r} \| \nabla f(u_{k+1}) - \nabla f(u_{k})\|^2 \\ + (r - 1) \left( \|M^\top( \lambda_{k} - \lambda_{k-1}) \|^2 - \|M^\top( \lambda_{k+1} - \lambda_k) \|^2 \right),
\end{multline*}
which implies
\begin{align*}
\frac{1}{r\rho} \| \lambda_{k+1}\! -\! \lambda_k \|^2 &\leq \frac{r - 1}{(2 - r) r\rho  \sigma_{\min}\left(M^\top M\right)} \left( \left\|M^\top( \lambda_k - \lambda_{k-1}) \right\|^2 \!-\! \left\|M^\top( \lambda_{k+1} - \lambda_k)\right \|^2 \right)\\
&\quad +\frac{r L_{\nabla f}^2 }{(2 - r)^2 \rho  \sigma_{\min}\left(M^\top M\right)} \| u_{k+1} \!-\! u_k \|^2 .
\end{align*}
Combining both cases proves the result with the constants $C_1$ and $C_2$ defined above. 
\end{proof}

Combining Lemmas~\ref{lem:3} and \ref{lem:4}, we arrive at
\begin{align}\label{Equ:18}
&\mathscr{L}_\rho\left(z_{k+1},c_{k+1},y_k\right) \!-\! \mathscr{L}_\rho\left(z_k,c_k,y_{k-1}\right) \notag \\
&\leq -\frac{\tau-2L_{{h}/{g^2}}L_g-L_\psi}{2}\|x_{k+1} - x_k\|^2-\!C_2\!\left(\|M^\top(\lambda_{k+1} \!-\! \lambda_{k})\|^2 \!-\! \|M^\top(\lambda_{k} \!-\! \lambda_{k-1})\|^2\right) \notag \\
&\quad -\frac{\rho\sigma_{\min} (M^\top M)-L_{\nabla f}- 2C_1}{2}\|u_{k+1} - u_k\|^2 .
\end{align} 
However, we observe that it is difficult to establish the descent 
property of the reformulated AL function 
$\mathscr{L}_{\rho}(x,u,\lambda,c,y)$. 
To overcome this difficulty, we construct the following Lyapunov function: 
$
P(x,u,\lambda,c,y,\bar{\lambda}) 
= \mathscr{L}_{\rho}(x,u,\lambda,c,y) 
+2C_2\big\|M^\top(\lambda - \bar{\lambda})\big\|^2 
$.
We evaluate the Lyapunov function along the iterates $\{z_k\}_{k\geq 1}$ generated by Algorithm~\ref{alg:1} and set $P_k = P(z_k, c_k, y_{k-1},\lambda_{k-1})$ and $Z_k = (z_k, c_k, y_{k-1}, \lambda_{k-1})$, where $c_k = 1/g(x_k)$ and $y_{k-1} \in \partial g(x_{k-1})$ for all $k\geq 1$.

\begin{assumption}\label{1129ass:1}
The parameters $(\tau,\rho)$ of Algorithm~\ref{alg:1} are subject to the following conditions:
\begin{enumerate}[label={(\roman*)}]
\item $\tau$ satisfies $\tau>2L_{h/g^2}L_g+L_\psi$.
\item We define the parameter-dependent quantity $\rho_0$ as   \begin{align*}\begin{cases}
\frac{\left(1+\sqrt{1+\frac{16}{r}}\right)L_{\nabla f}}{2\sigma_{\min}(M^\top M)}, & 0 < r \leq 1, \\[4pt]
\frac{\left(1+\sqrt{1+\frac{16r}{(2-r)^2}}\right)L_{\nabla f}}{2\sigma_{\min}(M^\top M)}, & 1 < r < 2,
\end{cases}\end{align*} and require that $\rho > \rho_0$.
\end{enumerate}
\end{assumption}

\begin{lemma}\label{lem:6}
Suppose that Assumptions~\ref{ass:1}, \ref{0902ass:1}, and~\ref{1129ass:1} hold. Then for all $ k \geq 1 $, the Lyapunov function $ P_k $ satisfies
\begin{align*}
P_{k+1} - P_k \leq -T_1\|x_{k+1} - x_k\|^2 - T_2\|u_{k+1} - u_k\|^2 - T_3\|\lambda_{k+1} - \lambda_k\|^2,
\end{align*}
where 
\[ T_1 = \frac{\tau-2L_{{h}/{g^2}}L_g-L_\psi}{2} ,\,  T_2 = \frac{\rho\sigma_{\min}(M^\top M)-L_{\nabla f} -4C_1}{2} ,\,  T_3 =\frac {1}{r\rho}. \] It follows that the sequence $ \{P_k\}_{k \geq 1} $ is nonincreasing.
\end{lemma}

\begin{proof}
From \cref{Equ:18} and the definition of $P_k$, we immediately obtain  
\begin{align*}
&P_{k+1} - P_k\\
&= \mathscr{L}_\rho(z_{k+1},c_{k+1},y_k)  + 2C_2\left\|M^\top(\lambda_{k+1} - \lambda_k)\right\|^2 - \mathscr{L}_\rho(z_{k},c_k,y_{k-1}) \\
&\quad  - 2C_2\left\|M^\top(\lambda_{k} - \lambda_{k-1})\right\|^2 \\
&\leq -\frac{\tau-2L_{{h}/{g^2}}L_g-L_\psi}{2}\|x_{k+1} - x_k\|^2 -\frac{\rho\sigma_{\min}(M^\top M)-L_{\nabla f} -4C_1}{2}\|u_{k+1}-u_k\|^2\\
&\quad -C_1\|u_{k+1}-u_k\|^2-C_2\left(\left\|M^\top(\lambda_{k}-\lambda_{k-1})\right\|^2-\left\|M^\top(\lambda_{k+1}-\lambda_{k})\right\|^2\right).
\end{align*}
By virtue of Lemma~\ref{lem:4}, we establish the descent property of $P_k$. Under Assumption~\ref{1129ass:1}, we have $T_1, T_2 > 0$, and so, for all $k \geq 1$, $P_{k+1} - P_k \leq 0$.
\end{proof}

\begin{lemma}\label{0611lem:3}
Suppose that Assumptions~\ref{ass:1} and~\ref{1129ass:1} hold. Assume that for every multiplier value attained by the algorithm, the level set $\mathcal{Q}(\lambda) :=\{(x,u)\in \mathcal{S}\times\mathbb R^m: \mathcal{L}_\rho(x,u,\lambda)\leq \mathcal{L}_\rho(x_0,u_0,\lambda_0)\}$ is bounded with a bound uniform over those multiplier values. If $\{\lambda_k\}$ is bounded, $M^\top\lambda_0 =\nabla f(u_0)$, and $\tau >0$ is sufficiently large, then $\{(x_k,u_k)\}$ is bounded.
\end{lemma}

\begin{proof}
Set $y_{-1} =y_0\in \partial g(x_0)$, $\lambda_{-1} =\lambda_0$, and $P_0 :=P(z_0,c_0,y_{-1},\lambda_{-1})$. Then $P_0 =\mathcal{L}_\rho(x_0,u_0,\lambda_0)$. Moreover, $M^\top\lambda_0 =\nabla f(u_0)$ implies $M^\top\lambda_0 =r\nabla f(u_0) +(1 -r)M^\top\lambda_{-1}$. Thus, the multiplier recurrence and the resulting estimate in the proof of Lemma~\ref{lem:4} also hold for $k =0$.

Let $\mathcal{Q}_{x,u} :=\operatorname{cl}\big(\bigcup_{k\geq0}\mathcal{Q}(\lambda_k)\big)$. By the uniform boundedness of the level sets $\mathcal{Q}(\lambda_k)$, $\mathcal{Q}_{x,u}$ is compact. We prove by induction that $(x_k,u_k)\in \mathcal{Q}_{x,u}$ and $P_0\geq P_1\ge\cdots\geq P_k$. The assertion is immediate for $k =0$, since $(x_0,u_0)\in \mathcal{Q}(\lambda_0)\subseteq \mathcal{Q}_{x,u}$.

Suppose that the assertion holds up to some $k\geq0$. By \cref{1129Equ3.5} and expanding $\widetilde\psi$, it follows that $c_kh(x_{k+1})-\langle c_k^2h(x_k)y_k,x_{k+1}-x_k\rangle+\langle \lambda_k,\ell_F(x_{k+1};x_k)-Mu_k\rangle+\frac{\rho}{2}\big\|\ell_F(x_{k+1};x_k)-Mu_k\big\|^2+D_{\phi_k} (x_{k+1},x_k)\leq c_kh(x_k)+\psi(z_k)$.
Since $D_{\phi_k}(x_{k+1},x_k)\geq \frac{\tau}{2}\|x_{k+1}-x_k\|^2$ and $\langle c_k^2h(x_k)y_k,x_{k+1}-x_k\rangle\leq \|c_k^2h(x_k)y_k\|\cdot\|x_{k+1}-x_k\|$, we obtain
\begin{align*}
c_kh(x_{k+1})+&\frac{\rho}{2}\Big\|F(x_k)+\nabla F(x_k)(x_{k+1}-x_k)-Mu_k+\frac{\lambda_k}{\rho}\Big\|^2+\frac{\tau}{2}\|x_{k+1}-x_k\|^2\\
&-\|c_k^2h(x_k)y_k\|\cdot\|x_{k+1}-x_k\|\leq c_kh(x_k)+\psi(z_k)+\frac{\|\lambda_k\|^2}{2\rho}.
\end{align*}
Since $c_k h(x_{k+1}) \geq 0$, the first two terms on the left-hand side are nonnegative. Hence,
\begin{align*}
\frac{\tau}{2}\|x_{k+1}-x_k\|^2-\|c_k^2h(x_k)y_k\|\cdot\|x_{k+1}-x_k\|\leq c_kh(x_k)+|\psi(z_k)|+\frac{\|\lambda_k\|^2}{2\rho}.
\end{align*}
Setting $t_x:=\|x_{k+1}-x_k\|$, we obtain
\begin{align}\label{0905eq:1}
t_x \leq \frac{1}{\tau}\left(\|c_k^2 h(x_k) y_k\| + \sqrt{\|c_k^2 h(x_k) y_k\|^2 + 2\tau\left(c_k h(x_k)+\left|\psi(z_k)\right|+\frac{\|\lambda_k\|^2}{2\rho}\right)}\right).
\end{align}

Let $\mathcal{Q}^x$ and $\mathcal{Q}^u$ denote the projections of $\mathcal{Q}_{x,u}$ onto $\mathbb R^n$ and $\mathbb R^m$, respectively. Since $g$ is continuous and strictly positive on the compact set $\mathcal{Q}^x$, there exists a sufficiently small $\Delta>0$ such that $g$ is bounded away from zero on $\mathcal{Q}_\Delta^x:=\{x\in \mathbb R^n:\dist(x,\mathcal{Q}^x)^2\leq\Delta\}$ and $\mathcal{Q}_\Delta^x$ is contained in the neighborhood on which $h$ and $g$ are Lipschitz continuous. By the assumptions on $h$, $g$, and $F$, the corresponding constants in Lemmas~\ref{0902lem:1} and~\ref{0902lem:2} can be defined uniformly on the compact sets constructed from $\mathcal{Q}_\Delta^x$, $\mathcal{Q}^u$, and the bounded sequence $\{\lambda_k\}$. In particular, there exist $A,B>0$ such that $\|c_k^2 h(x_k)y_k\|\leq A$ and $c_k h(x_k) +|\psi(z_k)| +\frac{\|\lambda_k\|^2}{2\rho}\leq B$. Hence, $t_x \leq \frac{1}{\tau}\left(A+\sqrt{A^2+2\tau B}\right)$. If $\tau\geq\frac{2A}{\sqrt{\Delta}}+\frac{2B}{\Delta}$, then $\|x_{k+1}-x_k\|\leq\sqrt{\Delta}$, and hence $x_{k+1}\in \mathcal{Q}_\Delta^x$. Consequently, the estimates in the proofs of Lemmas~\ref{lem:3} and~\ref{lem:6} apply at iteration $k$, and we obtain $P_{k+1}\leq P_k$.

By the definition of $P_{k+1}$, $c_{k+1}=1/g(x_{k+1})$, and the Fenchel--Young inequality, $\mathcal{L}_\rho(x_{k+1},u_{k+1},\lambda_{k+1})\leq P_{k+1}\leq P_0=\mathcal{L}_\rho(x_0,u_0,\lambda_0)$. Thus, $(x_{k+1},u_{k+1})\in \mathcal{Q}(\lambda_{k+1})\subseteq \mathcal{Q}_{x,u}$. This completes the induction. Since $\mathcal{Q}_{x,u}$ is bounded, the sequence $\{(x_k,u_k)\}$ is bounded.
\end{proof}

\begin{remark}
Assumption~\ref{0902ass:1} is standard in nonconvex nonsmooth optimization \cite{bolte2018nonconvex,boct2017proximal,bourkhissi2025convergence}. The augmented Lagrangian’s coercivity or level-boundedness ensures bounded primal iterates. Dual iterate boundedness remains open \cite{hallak2023adaptive} and is not addressed in our work.
\end{remark}

\begin{lemma}[Bound for $\partial P(Z_{k+1})$]\label{lem:7} 
Suppose that Assumptions~\ref{ass:1} and \ref{0902ass:1} hold. Then, there exists a subgradient $D_{k+1} = (d_1, d_2, \dots, d_6) \in \partial P(Z_{k+1})$ satisfying
$
\|D_{k+1}\| \leq T_4 \|z_{k+1} - z_k\|,
$
where
\[
T_4:=\frac{3L_gL_h}{m_g^2}+L_{\nabla\phi}+L_gL_{h/g^2}+\rho M_{\nabla F}^2+L_\psi+M_{h/g^2}+4M_{h/g^2}M_gL_g+\|M\|+\frac{1}{r\rho}+8C_2\|MM^\top\|.
\]
\end{lemma}

\begin{proof}
First, we define
\(\Phi := c^2 h(x) \left( g^*(y) + {2}/{c} - \langle x, y \rangle \right).
\)
Then, the function \( P(x, u, \lambda, c, y, \bar{\lambda}) \) is expressed as
\[
P(x, u, \lambda, c, y, \bar{\lambda}) = \Phi+\delta_{\mathcal{S}}(x) + f(u) + \psi(x, u, \lambda) + 2C_2 \left\| M^\top (\lambda - \bar{\lambda}) \right\|^2,
\]
which leads to
\begin{align}\label{equ:1}
\partial P(Z_{k+1})=
\begin{pmatrix}
\partial_x (\Phi+\delta_{\mathcal{S}})(x) + \nabla_x \psi(x_{k+1}, u_{k+1}, \lambda_{k+1}) \\
\nabla_u \psi(x_{k+1}, u_{k+1}, \lambda_{k+1}) + \nabla f(u_{k+1})  \\
\nabla_\lambda \psi(x_{k+1}, u_{k+1}, \lambda_{k+1}) + 4C_2 M M^\top (\lambda_{k+1} - \lambda_k) \\
\partial_c \Phi \\
\partial_y \Phi\\
4C_2 M M^\top (\lambda_k - \lambda_{k+1})
\end{pmatrix}.
\end{align}
To make use of Lemma~\ref{lem:2}, we define three auxiliary functions: $\varphi_l := c^2 h(x)$, $\varphi_c := g^*(y)$, and $\varphi_r := 2/c - \langle x, y \rangle$. Direct substitution yields the expression $\Phi = \varphi_l(\varphi_c + \varphi_r)$. Relying on this decomposition together with Lemma~\ref{lem:2}, it immediately follows that
\begin{align}\label{equ:2}
\widehat\partial \Phi\supseteq
\begin{pmatrix}
\widehat\partial_x \Psi,\,
0,\,
0,\,
\widehat\partial_c \Psi,\,
\widehat\partial_y \Psi + c_{k+1}^2 h(x_{k+1}) \partial g^*(y_k),\,
0
\end{pmatrix}^\top,
\end{align}
where $\Psi$ is defined as $\Psi:=(g^*(y_k)+2/c_{k+1}-\langle x_{k+1},y_k \rangle)c^2h(x) +c_{k+1}^2h(x_{k+1})(2/c-\langle  x,y \rangle)$.
Note that $(g^*(y_k)+2/c_{k+1}-\langle x_{k+1},y_k \rangle)c_{k+1}^2\geq (2/c_{k+1}-g(x_{k+1}))c_{k+1}^2=c_{k+1}>0$.
By virtue of \cref{equ:1}, \cref{equ:2}, the regularity of $h$ and $\delta_{\mathcal{S}}$, 
as well as the subdifferential inclusion
$
\partial (\Phi+\delta_\mathcal{S})\supseteq \widehat\partial (\Phi+\delta_\mathcal{S})\supseteq \widehat\partial \Phi+\widehat\partial \delta_{\mathcal{S}}
$
derived from \cite[Corollary 10.9]{Rockafellar2004var}, 
we conclude that for $D_{k+1} = (d_1,d_2,\dots,d_6)\in \partial P(Z_{k+1})$, we have
\begin{align*}
\left\{\begin{aligned}
d_1\in& (2c_{k+1}+c_{k+1}^2(g^*(y_k)-\langle x_{k+1},y_k \rangle))\partial h(x_{k+1})-c_{k+1}^2h(x_{k+1})y_k\\
&+\partial \delta_\mathcal{S}(x_{k+1})+\nabla_{x}\psi(x_{k+1},u_{k+1},\lambda_{k+1}),\nonumber\\
d_2=&\nabla_u\psi(x_{k+1},u_{k+1},\lambda_{k+1})+\nabla f(u_{k+1}),\\
d_3=&F(x_{k+1})-Mu_{k+1}+4C_2MM^\top(\lambda_{k+1}-\lambda_{k}),\\
d_4=&2h(x_{k+1})+2c_{k+1}h(x_{k+1})(g^*(y_k)-\langle x_{k+1},y_k \rangle),\\
d_5\in&c_{k+1}^2h(x_{k+1})(\partial g^*(y_k)-x_{k+1}),\\
d_6=&4C_2MM^\top(\lambda_{k}-\lambda_{k+1}).
\end{aligned}\right.
\end{align*}
From the $x$-subproblem, there exists $d_h\in \partial h(x_{k+1})$ and $d_\delta\in \partial \delta_\mathcal{S}(x_{k+1})$ such that
$ d_h=\!-\!g(x_k)\left(\nabla_x \widetilde \psi(x_{k+1},u_{k},\lambda_{k};x_k)\!+d_\delta\!-\!c_k^2 h(x_k)y_k+\!\nabla {\phi_k}(x_{k+1})\!-\!\nabla{\phi_k}(x_k)\right)$. 
It follows that there exists $d_1$, such that
\begin{align*}
d_1 &= \left( 2c_{k+1} + c_{k+1}^2 \left( g^*(y_k) \!-\! \langle x_{k+1}, y_k \rangle \right) \right) d_h \! +\! \nabla_x \psi(x_{k+1}, u_{k+1}, \lambda_{k+1})\!-\! c_{k+1}^2 h(x_{k+1}) y_k +d_\delta\\
&= \left( 2c_{k+1} + c_{k+1}^2 \left( g^*(y_k) - \langle x_{k+1}, y_k \rangle \right)- c_k  \right) d_h + \nabla_x \psi(x_{k+1}, u_{k+1}, \lambda_{k+1}) \\
&\quad -\nabla_x \widetilde\psi(x_{k+1}, u_k, \lambda_k;x_k) -  (\nabla \phi_k(x_{k+1})\!-\!\nabla\phi_k(x_k)) \!+\! y_k \left( c_k^2 h(x_k) \!-\! c_{k+1}^2 h(x_{k+1}) \right).
\end{align*}
Note that $g^*(y_k) = -g(x_k) + \langle x_k, y_k \rangle$. It follows that
\begin{align}\label{1127eq:3}
&\left\| 2c_{k+1} + c_{k+1}^2 \left( g^*(y_k) - \langle x_{k+1}, y_k \rangle  \right)- c_k \right\| \notag \\
&= \left\| 2c_{k+1} - c_k - c_{k+1}^2 g(x_k) + \langle x_k - x_{k+1}, c_{k+1}^2 y_k \rangle \right\| \notag \\
&\leq \left\| 2c_{k+1} - c_k - c_{k+1}^2 g(x_k)\right\| +\left\| \langle x_k - x_{k+1}, c_{k+1}^2 y_k \rangle \right\| \notag \\
&\leq \left\| c_{k+1} - c_k \right\|+\left\|c_{k+1}^2g(x_{k+1})-c_{k+1}^2 g(x_k)\right\| +\left\| \langle x_k - x_{k+1}, c_{k+1}^2 y_k \rangle \right\| \notag \\
&\leq \frac{3 L_g}{m_g^2} \| x_{k+1} - x_k \|,
\end{align}
where the last inequality follows from Lemma~\ref{0902lem:1} and the fact that $\|y_k\| \leq L_g$. On the other hand, by the definition of $\nabla_x \widetilde\psi(x_{k+1}, u_k, \lambda_k;x_k) $, it holds that
\begin{align*}
&\left\|\nabla_x \psi(x_{k+1}, u_{k+1}, \lambda_{k+1}) 
- \nabla_x \widetilde\psi(x_{k+1}, u_k, \lambda_k;x_k)\right\|\\
&= \left\| \nabla_x \psi(x_{k+1}, u_{k+1}, \lambda_{k+1}) 
- \nabla_x \psi(x_{k}, u_{k}, \lambda_{k})-\rho\nabla F(x_k)^\top \nabla F(x_k)(x_{k+1}-x_k)   \right\|   \\
&\leq L_\psi\left\|z_{k+1}-z_k \right\|+\rho M_{\nabla F}^2\left\|x_{k+1}-x_k\right\|, 
\end{align*}
where the last inequality is due to Lemma~\ref{0902lem:1}~(ii) and Lemma~\ref{0902lem:2}. Moreover, $d_h \in \partial h(x_{k+1})$, and, as $h$ is Lipschitz continuous, it satisfies $\|d_h\| \leq L_h$. 
Combining the above results, \cref{1127eq:3}, and the uniform Lipschitz continuity of $\nabla \phi_k$ leads to
\begin{align}\label{Eq:d1}
\|d_1\|\leq \left(\frac{3L_g L_h}{m_g^2}+ L_{\nabla \phi}+L_gL_{h/g^2}\right)\|x_{k+1}-x_k\|+\rho M_{\nabla F}^2\|x_{k+1}-x_k\| +L_\psi\|z_{k+1}-z_{k}\|.
\end{align}
The optimality condition of \cref{Equ:10}, together with the definition of $d_2$, yields
\begin{align}\label{Eq:d2}
\|d_2\| = \left\| \nabla f(u_{k+1}) - M^\top \lambda_{k+1} - \rho M^\top(F(x_{k+1}) -M u_{k+1}) \right\| 
\leq \|M\| \|\lambda_{k+1} - \lambda_k\|,
\end{align}
Substituting $g^*(y_k) = -g(x_k) + \langle x_k, y_k \rangle$ into $d_4$ and recalling $\|y_k\| \leq L_g$, we get
\begin{align}\label{Eq:d4}
\| d_4 \| &= \left\| 2h(x_{k+1}) + 2c_{k+1}h(x_{k+1}) \left( \langle x_k, y_k \rangle - g(x_{k}) - \langle x_{k+1}, y_k \rangle \right) \right\| \notag \\
&= \left\| 2h(x_{k+1})c_{k+1} \left( g(x_{k+1}) - g(x_k) + \langle x_k - x_{k+1}, y_k \rangle \right) \right\| \notag \\
&\leq 4 |h(x_{k+1})c_{k+1}|\cdot  L_g \| x_{k+1} - x_k \| \notag \\
&\leq 4 M_{h/g^2} M_g L_g \| x_{k+1} - x_k \|.
\end{align}
Since $y_k\in \partial g(x_k)$, it holds that $x_k\in \partial g^*(y_k)$, which implies $d_5=c_{k+1}^2h(x_{k+1})(x_{k}-x_{k+1})$, and
\begin{align}\label{eq:d5}
\|d_5\|\leq M_{h/g^2}\|x_{k+1}-x_k\|.
\end{align}
The remaining subgradient components follow directly from their definitions, i.e., 
\begin{align}
\|d_3\| &\leq \left( \frac{1}{r\rho} + 4C_2\|MM^\top\| \right) \| \lambda_{k+1} - \lambda_k \|,\, 	\|d_6\| \leq 4C_2\|MM^\top\| \| \lambda_{k+1} - \lambda_k \|. \label{Eq:d7}
\end{align}
Combining the estimates from \cref{Eq:d1,Eq:d2,Eq:d4,eq:d5,Eq:d7} and the triangle inequality, we obtain
$
\|D_{k+1}\| \leq \sum_{i=1}^6 \|d_i\| \leq T_4 \|z_{k+1} - z_k\|,
$
where $T_4$ is defined as
$
T_4 = {3L_gL_h}/{m_g^2}+L_{\nabla\phi}+L_gL_{h/g^2}+\rho M_{\nabla F}^2+L_\psi+M_{h/g^2}+4M_{h/g^2}M_gL_g+\|M\|+{1}/{r\rho}+8C_2\|MM^\top\|
$.
\end{proof}

\begin{theorem}\label{the:1}  
Suppose that Assumptions~\ref{ass:1}, \ref{0902ass:1} and \ref{1129ass:1} hold.
Then the following statements hold:
\begin{enumerate}
\item The following nonnegative series converge (i.e., are finite): 
\begin{equation}
    \sum_{k=1}^{+\infty} \|x_{k+1} - x_k\|^2 < +\infty, \,
\sum_{k=1}^{+\infty} \|u_{k+1} - u_k\|^2 < +\infty,  \,
\sum_{k=1}^{+\infty} \|\lambda_{k+1} - \lambda_k\|^2 < +\infty. 
\end{equation}

\item If $\{z_{k_q}\}_{q\geq 1}$ is a subsequence of $\{z_k\}_{k\geq 1}$ converging to $z^*=(x^*, u^*, \lambda^*)$, then there exists a further subsequence $\{k_{q_j}\}_{j\geq 1}$ of $\{k_q\}$ and $y^*\in \partial g(x^*)$ such that $\{Z_{k_{q_j}}\}_{j\geq 1}\to Z^*=(z^*,c^*,y^*,\lambda^*)$ with $c^*=1/g(x^*)$, and the Lyapunov function satisfies $P_{k_{q_j}}\to P(Z^*)$. Furthermore, $\{P_k\}_{k\geq 1}$ converges with $P_k\to P^*$ as $k\to +\infty$, where $P^*:=P(Z^*)$.

\item Let $\omega(\{z_k\}_{k\geq 1})$ denote the set of cluster points of $\{z_k\}_{k\geq 1}$, and define
$\Omega := \omega\left( \{Z_k\}_{k\geq 1} \right)$.
We then have $\Omega \subseteq \crit(P)$, where $\crit(\cdot)$ denotes the set of critical points of the corresponding mapping.
Consequently, any cluster point of the sequence $\{z_k\}_{k\geq 1}$ is a lifted critical point of problem~\cref{Equ:1}.
\item $\Omega$ is compact and $P$ is constant on $\Omega$, with this constant value equal to $P^* = \lim_{k\to+\infty} P_k$.
\end{enumerate}
\end{theorem}

\begin{proof}
(i) By the Fenchel--Young inequality and the definition of $P_k$, we have \begin{equation}\label{1008eq:7}
    P_k\geq f(u_k) +h(x_k)/g(x_k) +\psi(z_k)\geq f(u_k) +\psi(z_k).
\end{equation} 
Since $f$ and $\psi$ are continuous and $\{z_k\}$ is bounded by Assumption~\ref{0902ass:1}, together with \cref{1008eq:7}, it follows that $\inf_{k\geq1} P_k >-\infty$. By the descent property of $P_k$ in Lemma~\ref{lem:6}, summing over $k=1$ to $K$ gives
\[\begin{aligned}
    &T_1\sum_{k=1}^{K} \|x_{k+1} - x_k\|^2 + T_2\sum_{k=1}^{K} \|u_{k+1} - u_k\|^2 + T_3\sum_{k=1}^{K} \|\lambda_{k+1} - \lambda_k\|^2 
    \leq P_1 - P_{K+1}\leq P_1-\inf_{k\geq1} P_k.
\end{aligned}\]
Letting $K \to +\infty$, and using $T_1, T_2, T_3>0$ proves (i).

(ii) By Assumption~\ref{0902ass:1}, the sequence $\{z_k\}_{k\geq 1}$ admits at least one subsequence $\{z_{k_q}\}_{q\geq 1}$ converging to $z^*=(x^*,u^*,\lambda^*)$. Furthermore, by Assumption~\ref{ass:1}~(v), the definition of $c$, and the continuity of $g$ (which follows from its Lipschitz continuity), we have $c_{k_q}\to c^*$ with $c^*=1/g(x^*)$. Since $g$ is Lipschitz continuous, every $y\in \partial g(x)$ satisfies $\|y\|\leq L_g$. Therefore, we can extract a further subsequence $\{k_{q_j-1}\}_{j\geq 1}$ from $\{k_q\}_{q\geq 1}$ such that $y_{k_{q_j-1}}$ converges to some limit $y^*$. Moreover, since $x_{k_{q_j-1}}\to x^*$ and $y_{k_{q_j-1}}\in \partial g(x_{k_{q_j-1}})$, the closedness property of the limiting subdifferential guarantees that $y^*\in \partial g(x^*)$. Combining the above facts together with (i), there exists $y^*\in \partial g(x^*)$ such that $\{Z_{k_{q_j}}\}\to Z^*=(z^*,c^*,y^*,\lambda^*)$ with $c^*=1/g(x^*)$.
Note that
\begin{align}\label{0809eq:1}
P_{k_{q_j}} &= c^2_{k_{q_j}}h(x_{k_{q_j}})\big(g^*(y_{k_{q_j}-1})+2/c_{k_{q_j}}-\langle x_{k_{q_j}},y_{k_{q_j}-1}\rangle\big)+\delta_\mathcal{S}(x_{k_{q_j}})+f(u_{k_{q_j}}) \notag \\
&\quad +\psi(x_{k_{q_j}},u_{k_{q_j}},\lambda_{k_{q_j}})+2C_2\big\|M^\top (\lambda_{k_{q_j}}-\lambda_{k_{q_j}-1})\big\|^2,
\end{align}
By \cref{1129Equ3.5} and the construction of the subproblem, we have $x_{k_{q_j}}\in \mathcal{S}$, so that $\delta_\mathcal{S}(x_{k_{q_j}})=0$ for all $j$.
Moreover, using the Fenchel–Young identity
$g^*(y_{k_{q_j}-1})=-g(x_{k_{q_j}-1})+\langle x_{k_{q_j}-1},y_{k_{q_j}-1}\rangle$,
together with the continuity of $g$, we obtain
\begin{align*}
g^*(y_{k_{q_j}-1})+\frac{2}{c_{k_{q_j}}}-\langle x_{k_{q_j}},y_{k_{q_j}-1}\rangle
&=-g(x_{k_{q_j}-1})+\frac{2}{c_{k_{q_j}}}+\langle x_{k_{q_j}-1}- x_{k_{q_j}},y_{k_{q_j}-1}\rangle\\
&\to -g(x^*)+\frac{2}{c^*},\quad \text{as }j\to+\infty.
\end{align*}
Combined with the continuity of $h$, the first term on the right‑hand side of \cref{0809eq:1} converges to $c^* h(x^*)$.
The limits of the remaining terms follow from the respective continuity assumptions.
Consequently,
\begin{align*}
\lim_{j\to +\infty}P_{k_{q_j}}
= c^* h(x^*) + f(u^*)+\psi(x^*,u^*,\lambda^*)
= P(z^*,c^*,y^*,\lambda^*).
\end{align*}
In other words, the monotonic sequence $\{P_k\}_{k\geq 1}$ possesses a convergent subsequence $\{P_{k_{q_j}}\}_{j\geq 1}$.
By a standard property of monotone real sequences, the whole sequence $\{P_k\}_{k\geq 1}$ is convergent, and hence $P_k \to P^*$ as $k \to +\infty$.

(iii) For the sequence $\{D_{k}\}$ defined in Lemma~\ref{lem:7}, we have
$D_{k_{q_j}}\in \partial P(Z_{k_{q_j}})$ for every $j\geq 1$, $D_{k_{q_j}}\to 0$ as $j\to+\infty$,
$Z_{k_{q_j}}\to Z^*$, and $P(Z_{k_{q_j}})\to P(Z^*)$ as $j\to+\infty$.
The closedness criterion for the limiting subdifferential guarantees that $0\in \partial P(Z^*)$,
which implies $\Omega\subseteq \crit(P)$.
Combined with the asymptotic property of $\lambda$, we further obtain $(z^*,c^*,y^*)\in \crit(\mathscr{L}_{\rho})$.
By virtue of Lemma~\ref{lemma:2}, $z^*$ is a lifted critical point of problem~\cref{Equ:1}.
Therefore, every cluster point of $\{z_k\}_{k\geq 1}$ is a lifted critical point of problem~\cref{Equ:1}.

(iv) The boundedness of $\{z_k\}$, $\{c_k\}$, $\{y_k\}$, and $\{\lambda_{k-1}\}$ implies that $\{Z_k\}$ is bounded; hence its cluster set $\Omega$ is nonempty and compact. By~(ii), every $Z^*\in \Omega$ satisfies $P(Z^*)=P^*$, so $P$ is constant on $\Omega$.
\end{proof}

\subsection{Convergence under the KL property} \label{Sect3.3} 

In this subsection, we establish the global convergence and convergence rates of the sequence $\{z_k\}$ generated by Algorithm~\ref{alg:1} under the KL property. 

\begin{theorem}\label{the:klconvergence}
Suppose that Assumptions~\ref{ass:1}, \ref{0902ass:1} and \ref{1129ass:1} hold. Then the following statements hold:
\begin{enumerate}
\item Suppose that the Lyapunov function $P$ satisfies the KL property at each point of $\Omega$. Then the sequence $\{z_k\}_{k \geq 1}$ converges to a lifted critical point $z^*$ of the problem \eqref{Equ:1} and $\sum_{k=1}^{+\infty}\|z_k-z_{k-1}\|<+\infty$.
\item Suppose, moreover, that the KL property holds uniformly around $\Omega$ with exponent $\theta\in[0,1)$ and desingularizing function
$\varphi(t)=\frac{c_l}{1-\theta}t^{1-\theta}$. Let $\varepsilon_k:=P_k-P^*$. Then:
\begin{enumerate}
\item If $\theta = 0$, then $\{z_k\}$ converges to $z^*$ in a finite number of iterations.

\item If $\theta \in \left(0, \frac{1}{2}\right)$, then  $\{\varepsilon_k\}$ converges $Q$-superlinearly to $0$ and $\{z_k\}$ converges $R$-superlinearly to $z^*$ as $k \to +\infty$.

\item If $\theta = \frac{1}{2}$, then $\{\varepsilon_k\}$ converges $Q$-linearly to 0 and $\{z_k\}$ converges $R$-linearly to $z^*$ as $k \to +\infty$.

\item If $\theta \in \left(\frac{1}{2}, 1\right)$, then $\{\varepsilon_k\}$ converges sublinearly to 0 and $\{z_k\}$ converges sublinearly to $z^*$ as $k \to +\infty$.
\end{enumerate}
\end{enumerate}
\end{theorem}

\begin{proof}
By Lemma~\ref{0624lem:1}, Lemma~\ref{lem:6}, Lemma~\ref{lem:7}, and Theorem~\ref{the:1}, the full sequential convergence (i) and linear convergence (ii)--c under the KL property follow from Bo{\c{t}} et al. \cite[Theorem 5.1]{boct2021extrapolated}, while the finite (ii)--a and sublinear (ii)--d convergence results follow from the arguments in the proof of Attouch and Bolte \cite[Theorem 2]{attouch2009convergence}. $\{z_k\}$ converges to a lifted critical point of problem \eqref{Equ:1}, as stated in \cref{the:1}~(iii).
With Lemma~\ref{lem:6} and Lemma~\ref{lem:7}, we have 
\[P_{k-1}-P_k
\geq \mu \|z_k-z_{k-1}\|^2
\geq \frac{\mu}{T_4^2}\dist^2\bigl(0,\partial P(Z_k)\bigr),\] 
where $\mu=\min\{T_1,T_2,T_3\}$.
Moreover, by the definition of the desingularizing function $\varphi$ and the KL property, it holds that
$
\dist(0,\partial P(Z_k)) \geq c_l^{-1}\,(P_k-P^*)^{\theta}$ and $P_k>P^*$.
Combining the above inequalities yields
$
\varepsilon_k - \varepsilon_{k-1} \leq -\left({\mu}/{T_4^2 c_l^2}\right)\cdot \varepsilon_k^{2\theta}$. This inequality satisfies the conditions in Lemma~\ref{lem:15}. We deduce that the sequence $\{\varepsilon_k\}$  satisfies $\lim_{k\to+\infty} \frac{\varepsilon_{k+1}}{\varepsilon_k}=0$,
which follows from the argument in the proof of Lemma~\ref{lem:15}. According to \cite[Chapter~9.1]{OrtegaRheinboldt1970}, the latter condition implies that $\{\varepsilon_k\}$ is $Q$-superlinearly convergent to $0$. Using the concavity of $\varphi$ and Lemma~\ref{lem:7}, we obtain that, for all sufficiently large $k\in \mathbb{N}$,
\[
\frac{\|z_{k+1}-z_k\|^2}{\|z_k-z_{k-1}\|}
\le
\frac{T_4}{\mu}[
\varphi(P_k-P^*)
-
\varphi(P_{k+1}-P^*)].\]
Let $\xi\in(0,1)$ be fixed. Following similar arguments as in the proof of Attouch and Bolte \cite[Theorem 1]{attouch2009convergence}, we obtain the following result:
\begin{align}\label{0126eq:7}
\sum_{j=k}^{k+l}\|z_{j+1}-z_j\|
\le
\frac{\xi}{1-\xi}\|z_k-z_{k-1}\|
+
\frac{T_4c_l}{\mu \xi(1-\theta)(1-\xi)}
\Big(
\varepsilon_k^{1-\theta}-\varepsilon_{k+l+1}^{1-\theta}
\Big).
\end{align}
Taking the limit $l\to+\infty$ in \cref{0126eq:7}, along with $z_k\to z^*$ and Lemma~\ref{lem:6}, we deduce
\begin{align*}
\|z_k-z^*\|\leq
\sum_{j=k}^{+\infty}\|z_{j+1}-z_j\|
\le
\frac{\xi}{(1-\xi)\sqrt{\mu}}\sqrt{\varepsilon_{k-1}}
+
\frac{T_4c_l}{\mu \xi(1-\theta)(1-\xi)}
\varepsilon_k^{1-\theta}.
\end{align*}
Since $\theta \in (0, \tfrac{1}{2})$, we have $\sqrt{\varepsilon_k} \geq \varepsilon_k^{1-\theta}$. Combining this with the monotonicity of $\{\varepsilon_k\}$, we obtain
\begin{align*}
\|z_k - z^*\| \leq \left( \frac{\xi}{(1-\xi)\sqrt{\mu}} + \frac{T_4 c_l}{\mu \xi (1-\theta)(1-\xi)} \right) \sqrt{\varepsilon_{k-1}}.
\end{align*}
Together with the convergence rate of $\{\varepsilon_k\}$, this completes the proof.
\end{proof}

\subsection{Convergence under the HVP-EB condition}\label{sec3.4} 

In this subsection, we use the HVP-EB condition with exponent $\theta \in [0,1)$ to 
ensure the global convergence and convergence rates of the sequence $\{z_k\}$. Moreover, to illustrate that the HVP-EB condition is a more 
natural and general setting, we note that the well-known KL property implies the HVP-EB condition, but the converse does not hold.

\begin{definition}[HVP-EB condition] \label{def:hvp}
We say that the HVP-EB condition holds uniformly along the generated sequence in a neighborhood of $\Omega$ with exponent $\theta\in[0,1)$ if there exist $\kappa > 0$, $\eta > 0$, and $\nu > 0$ such that 
\begin{align*}\left[P_{k+1} - P^*\right]^\theta \leq \kappa\|z_{k+1} - z_k\|,\end{align*}
when $Z_{k+1}\in \mathbb{B}({\Omega} ; \eta) \cap \left\{Z\in \mathbb{R}^n \!\times \!\mathbb{R}^m \!\times\! \mathbb{R}^q\!\times\! \mathbb{R} \!\times\! \mathbb{R}^n \!\times\!  \mathbb{R}^q: P^*\!<\!P(Z)\! < \!P^* \!+ \!\nu\right\}$.
\end{definition}

\begin{lemma}[The KL property implies the HVP-EB condition]\label{lem:100}
Suppose that Assumptions~\ref{ass:1}, \ref{0902ass:1} and \ref{1129ass:1} hold, and let $Z^* \in \Omega$.
If the function $P$ satisfies the KL property at the each point $Z^*\in \Omega$ with exponent
$\hat{\theta} \in [0,1)$ and parameters $\eta > 0$, $\nu > 0$, and
$\varphi(t) = \frac{c_l}{1-\hat{\theta}} t^{1-\hat{\theta}}$, 
then the HVP-EB conditionholds uniformly around $\Omega$ along the generated sequence with exponent $\theta = \hat{\theta}$.
\end{lemma}

\begin{proof}
Suppose that $P$ satisfies the KL property at each point $Z^*\in \Omega$ with exponent
$\hat{\theta} \in [0,1)$ and associated parameters $\eta > 0$, $\nu > 0$, and
$\varphi(t) = \frac{c_l}{1-\hat{\theta}} t^{1-\hat{\theta}}$.
Since $\Omega$ is compact and $P$ is identically equal to $P^*$ on $\Omega$,
for any iterate $Z_{k+1}$ such that $Z_{k+1} \in \mathbb{B}(\Omega; \eta)$ and $P^* < P(Z_{k+1}) < P^* + \nu$,
by Lemma~\ref{0624lem:1} and Lemma~\ref{lem:7}, it follows that
\begin{align*}
\bigl[P_{k+1} - P^*\bigr]^{\hat{\theta}}
\leq c_l \dist\bigl(0, \partial P(Z_{k+1})\bigr)
\leq c_l T_4 \| z_{k+1} - z_k \|.
\end{align*}
Setting $\theta = \hat{\theta} \in [0,1)$ yields $[P_{k+1} - P^*]^{\theta}\leq c_l T_4 \| z_{k+1} - z_k \|$.
Thus, we conclude that the KL property of $P$ at each point $Z^*\in \Omega$
implies the HVP-EB condition along the generated sequence near $\Omega$ with exponent $\theta = \hat{\theta}$.
\end{proof}

\begin{lemma}\label{the:5}
Suppose that Assumptions~\ref{ass:1}, \ref{0902ass:1} and \ref{1129ass:1} hold, $Z^* \in \Omega$,
and that the HVP-EB condition holds uniformly around $\Omega$ along the generated sequence with parameters
$\eta > 0$, $\nu > 0$, $\kappa > 0$, and $\theta \in [0,1)$.
Then there exists an integer $k_0  \geq 1$ such that, for all
$k \geq k_0$, it holds that 
$P_k < P^* + \nu$ and ${Z}_k \in \mathbb{B}(\Omega; \eta)$. If $\theta\in(0,1)$, furthermore, $z_k \in \mathbb{B}(z^*; \eta)$.
\end{lemma}

\begin{proof}
Since $\{Z_k\}$ is bounded and $\Omega$ is its set of accumulation points, we have $\dist(Z_k,\Omega)\to0$. Together with $P_k\to P^*$, for the given $\eta>0$ and $\nu>0$ there exists $k_1>0$ such that for $k\geq k_1$,
\begin{align*}
Z_k\in \big\{Z\in \mathbb{R}^n \!\times \!\mathbb{R}^m \!\times\! \mathbb{R}^q\!\times\! \mathbb{R} \!\times\! \mathbb{R}^n \!\times\!  \mathbb{R}^q: \dist(Z,\Omega)<\eta,\ P(Z)<P^*+\nu\big\},
\end{align*}
which means that for all $k \geq k_1$, it holds that 
$P_k < P^* + \nu$ and $Z_k \in \mathbb{B}(\Omega; \eta)$. If $\theta =0$, the first conclusion follows by taking $k_0 =k_1$. In the remainder of the proof, let $\theta\in (0,1)$.

By Lemma~\ref{lem:6}, the sequence $\{P_k\}_{k \geq 1}$ is nonincreasing and satisfies $P_k \geq P^*$ for all $k$. Combined with the facts that
$P_k \to P^*$ and $\|z_{k+1} - z_k\| \to 0$ as $k \to +\infty$, there exists
$k_2\geq k_1 >0$ such that
\begin{align}\label{Eq:34}
\|z_{k_2} - z^*\| + 2\sqrt{\frac{P_{k_2} - P^*}{\mu}}
+ \frac{\kappa}{(1 - \theta)\mu}(P_{k_2} - P^*)^{1 - \theta} < \eta,
\end{align}
and $P_{k_2} < P^* + \nu$.

We use induction to show that the sequence $\{z_k\}$ is contained in
$\mathbb{B}(z^*; \eta)$ for all $k > k_2$. By \cref{Eq:34}, it follows directly
that $z_{k_2} \in \mathbb{B}(z^*; \eta)$. 
Moreover, the inequalities $P^* \leq P_{k_2 + 1} \leq P_{k_2} < P^* + \nu$ hold trivially. 
Invoking Lemma \ref{lem:6}, we further obtain
\[\|z_{k_2+1} - z_{k_2}\| \leq \sqrt{\frac{P_{k_2} - P_{k_2 + 1}}{\mu}} 
\leq \sqrt{\frac{P_{k_2} - P^*}{\mu}},\]
and hence
\[\|z_{k_2 + 1} - z^*\| \leq \|z_{k_2} - z^*\| + \|z_{k_2+1} - z_{k_2}\| 
\leq \|z_{k_2} - z^*\| + \sqrt{\frac{P_{k_2} - P^*}{\mu}} 
\stackrel{\cref{Eq:34}}{<} \eta,
\]
where $\mu=\min\{T_1,T_2,T_3\}$. Thus, $z_{k_2 + 1} \in \mathbb{B}(z^*; \eta)$.
Now, for the induction hypothesis, assume that for all $j \in \{k_2 + 1, \dots, k_2 + k\}$,
$z_j \in \mathbb{B}(z^*; \eta)$ and that $z_{k_2 + k} \neq z_{k_2 + k + 1}$.
Since $\{P_k\}$ is nonincreasing, it follows that
$P^* < P_{k_2 + k + 1} \leq P_{k_2 + k} \leq \dots \leq P_{k_2 + 1} < P^* + \nu$.
To complete the inductive step, it suffices to show that $z_{k_2 + k + 1} \in \mathbb{B}(z^*; \eta)$. If $P_j =P^*$ for some $j\geq k_2 +1$, then by Lemma~\ref{lem:6}, $z_{l+1} =z_l$ for all $l\geq j$, and the conclusion follows. Hence, we may assume that $P_j >P^*$ for all $j\geq k_2 +1$. 
First, by the concavity of the function $x^a$ (with $a \in (0,1)$), we have
\begin{align}\label{Eq:38}
(P_j - P^*)^a - (P_{j + 1} - P^*)^a 
&\geq \frac{a (P_j- P_{j + 1})}{(P_j - P^*)^{1 - a}} 
\stackrel{{\rm Lemma~\ref{lem:6}}}{\geq} \frac{a \mu \|z_{j+1} - z_{j}\|^2}{(P_j - P^*)^{1 - a}},
\end{align}
for all $j \in \{k_2 + 1, \dots, k_2 + k\}$. Since $Z_j \in \mathbb{B}(\Omega; \eta)$,
the HVP-EB condition implies
\begin{align}\label{Eq:36}
(P_j - P^*)^\theta \leq \kappa \|z_{j} - z_{j-1}\|.
\end{align}
Since $P_j >P^*$, we have $\|z_j -z_{j-1}\| >0$.
The combination of \cref{Eq:38} (with $a = 1 - \theta$) and \cref{Eq:36} yields 
\[\frac{\kappa}{(1 - \theta)\mu} \|z_j - z_{j-1}\| \left[(P_j - P^*)^{1 - \theta} - (P_{j+1} - P^*)^{1 - \theta}\right] 
\geq \|z_{j+1} - z_{j}\|^2.\]
Consequently,
\begin{align}\label{Eq:39}
2\|z_{j+1} - z_{j}\| 
&\leq \|z_{j} - z_{j-1}\| + \frac{\|z_{j+1} - z_{j}\|^2}{\|z_{j} - z_{j-1}\|} \notag \\
&\leq \frac{\kappa}{(1 - \theta)\mu} \left[(P_j - P^*)^{1 - \theta} - (P_{j+1} - P^*)^{1 - \theta}\right] + \|z_{j} - z_{j-1}\|.
\end{align}
Summing \cref{Eq:39} over $j \in \{k_2 + 1, \dots, k_2 + k\}$ gives
\begin{align}\label{Eq:40}
&\sum_{j=k_2+1}^{k_2+k} \|z_{j+1} - z_j\| + \|z_{k_2+k+1} - z_{k_2+k}\| \notag \\
&\leq \frac{\kappa}{(1 - \theta)\mu} \left[(P_{k_2+1} - P^*)^{1 - \theta} - (P_{k_2+k+1} - P^*)^{1 - \theta}\right] + \|z_{k_2+1} - z_{k_2}\|.
\end{align}
Finally, applying the triangle inequality together with \cref{Eq:40}, we obtain
\begin{align*}
\|z_{k_2+k+1} - z^*\| 
&\leq \|z_{k_2} - z^*\| + \|z_{k_2+1} - z_{k_2}\| + \sum_{j=k_2+1}^{k_2+k} \|z_{j+1} - z_{j}\| \\
&\leq \|z_{k_2} - z^*\| + \frac{\kappa}{(1 - \theta)\mu}(P_{k_2+1} - P^*)^{1 - \theta} + 2 \|z_{k_2+1} - z_{k_2}\| \\
&\leq \|z_{k_2} - z^*\| + 2\sqrt{\frac{P_{k_2} - P^*}{\mu}} + \frac{\kappa}{(1 - \theta)\mu}(P_{k_2} - P^*)^{1 - \theta} \stackrel{\cref{Eq:34}}{<} \eta.
\end{align*}
This establishes that $z_{k_2 + k + 1} \in \mathbb{B}(z^*; \eta)$, completing the inductive step. Letting $k_0 =k_2$ completes the proof.
\end{proof}

\begin{theorem}
Suppose that Assumptions~\ref{ass:1}, \ref{0902ass:1} and \ref{1129ass:1} hold.
Let the HVP-EB condition hold uniformly around $\Omega$ along the generated sequence with parameters $\eta>0$, $\nu>0$, $\kappa>0$, and $\theta\in[0,1)$. Then:
\begin{enumerate}
\item The sequence $\{z_k\}_{k \geq 1}$ converges to a lifted critical point $z^*$ of the problem \eqref{Equ:1} and $\sum_{k=1}^{+\infty}\|z_k-z_{k-1}\|<+\infty$.
\item Denote $\varepsilon_k := P_k - P^*$, where $P^* = \lim_{k \to + \infty} P_k$. Then the following hold:
\begin{enumerate}
\item If $\theta = 0$, then the sequence $\{z_k\}$ converges to $z^*$ in a finite number of iterations.
\item If $\theta \in \left(0, \frac{1}{2}\right)$, then $\{\varepsilon_k\}$ converges $Q$-superlinearly to $0$ and $\{z_k\}$ converges $R$-superlinearly to $z^*$ as $k \to +\infty$.
\item If $\theta = \frac{1}{2}$, then $\{\varepsilon_k\}$ converges $Q$-linearly to $0$ and $\{z_k\}$ converges $R$-linearly to $z^*$ as $k \to +\infty$.
\item If $\theta \in \left(\frac{1}{2}, 1\right)$, then $\{\varepsilon_k\}$ converges sublinearly to 0 and $\{z_k\}$ converges sublinearly to $z^*$ as $k \to +\infty$.
\end{enumerate}
\end{enumerate}
\end{theorem}
\begin{proof}
(i) By Lemma~\ref{the:5}, there exists $k_0\geq1$ such that $P_k<P^*+\nu$ and $Z_k\in\mathbb{B}(\Omega;\eta)$ for all $k\geq k_0$. If there exists a finite integer $K\geq k_0$ such that $P_K=P^*$, then the nonincreasing property of $\{P_k\}$ gives $P_k=P^*$ for all $k\geq K$. From the sufficient-descent inequality given in Lemma~\ref{lem:6}, it follows that $z_{k+1}=z_k$ for all $k\geq K$. This yields the finite-convergence property of the sequence $\{z_k\}$.

Suppose that $P_k>P^*$ for all $k\geq k_0$. If $\theta=0$, the HVP--EB condition gives
$1 \leq \kappa \|z_k-z_{k-1}\|$ for all $k>k_0$. This contradicts Theorem~\ref{the:1}~(i), which ensures that
$\|z_k-z_{k-1}\|\to 0$ as $k\to+\infty$. Therefore, finite convergence must occur.

For $\theta \in (0,1)$, the finite-length property follows by taking the limit as $k \to +\infty$ in \cref{Eq:40}. Specifically, this yields
\begin{align}\label{Eq:41}
\sum_{j=k_0+1}^{+\infty} \|z_{j+1} - z_j\| \leq \frac{\kappa}{(1 - \theta)\mu} (P_{k_0+1} - P^*)^{1 - \theta} + \|z_{k_0+1} - z_{k_0}\|.
\end{align}
The sequence $\{z_k\}$ is Cauchy and converges to a lifted critical point of problem \eqref{Equ:1}, as stated in \cref{the:1}~(iii).

(ii) By $P_k>P^*$ together with the HVP-EB condition and Lemma~\ref{lem:6}, it follows that
\begin{align}\label{0612eq:1}
\varepsilon_{k}^{2\theta} \leq \kappa^2 \|z_{k} - z_{k-1}\|^2 \leq \frac{\kappa^2}{\mu} (P_{k-1} - P_{k}) = \frac{\kappa^2}{\mu} (\varepsilon_{k-1} - \varepsilon_{k}),\,\forall\,k> k_0.
\end{align}
By invoking \cref{Eq:41} and leveraging the convergence of $\{z_k\}$ to $z^*$, we establish that 
\begin{equation}\label{Eq:45}
    \begin{aligned}
\| z_k - z^* \| \leq \sum_{j=k}^{+\infty} \| z_{j+1} - z_j \| 
&\leq \frac{\kappa}{(1 - \theta)\mu} (P_{k} - P^*)^{1 - \theta} + \| z_{k} - z_{k-1} \| \\
&\leq \frac{\kappa}{(1 - \theta)\mu} \varepsilon_k^{1 - \theta} + \frac{1}{\sqrt{\mu}} \sqrt{\varepsilon_{k-1}}.
\end{aligned}
\end{equation}
Let $\theta = 0$. If $\varepsilon_k > 0$ for $k > k_0$, we have $0<\frac{\mu}{\kappa^2} \leq \varepsilon_{k-1} - \varepsilon_{k}$. As $k \to + \infty$, the right-hand side tends to zero, then $0<\frac{\mu}{\kappa^2} \le0$, which leads to a contradiction. Therefore, there exists a finite integer $K > k_0$ such that $\varepsilon_K = 0$, and $\varepsilon_k = 0$ for all $k \geq K$. This implies $\{\varepsilon_k\}$ converges to zero in finitely many iterations. Since $0 \leq \mu\|z_{k + 1} - z_k\|^2 \leq P_k - P_{k+1}$, for all $ k \in \mathbb{N}$, combining this inequality with the finite convergence of $\{\varepsilon_k\}$ establishes statement~(ii)--a. When $\theta \in (0, \frac{1}{2})$, by invoking \cref{0612eq:1} and Lemma \ref{lem:15}, we deduce that $\{\varepsilon_k\}$ converges $Q$-superlinearly to $0$, as in (ii)--(b) of Theorem~\ref{the:klconvergence}. 
When \(\theta = \frac12\), it follows from \cref{0612eq:1} that for all \(k > k_0\),
\begin{equation}\label{1008eq:8}
\varepsilon_{k} \leq \frac{1}{1 + \frac{\mu}{\kappa^2}} \varepsilon_{k-1}.
\end{equation}
Since \(\big(1 + {\mu}/{\kappa^2}\big)^{-1} \in (0, 1)\), together with \cref{1008eq:8}, the sequence \(\{\varepsilon_k\}\) converges linearly to 0 as \(k \to +\infty\).
When $\theta \in \left(\frac12,1\right)$, using \cref{0612eq:1}, the sublinear convergence analysis of $\{\varepsilon_k\}$ proceeds analogously to the arguments in \cite[Theorem 4]{boct2020proximal} and \cite[Theorem 3]{bourkhissi2025convergence}.

We now derive the convergence rate of $\{z_k\}$ for the corresponding $\theta$ from that of $\{\varepsilon_k\}$.
For $\theta \in (0, \frac{1}{2})$, we first observe that $\varepsilon_k^{1-\theta} \leq \varepsilon_k^{1/2}$ for $k > k_0$. Substituting this inequality into \cref{Eq:45}, we derive 
\begin{equation}\label{1008eq:9}
    \| z_k - z^* \| \leq \frac{\kappa}{(1 - \theta)\mu} \sqrt{\varepsilon_k} + \frac{1}{\sqrt{\mu}} \sqrt{\varepsilon_{k-1}}.
\end{equation}
Since $\{\varepsilon_k\}$ is nonincreasing, \cref{1008eq:9} implies 
\begin{equation}\label{1008eq:10}
    \| z_k - z^* \| \leq \frac{\kappa + (1-\theta)\sqrt{\mu}}{(1-\theta)\mu} \sqrt{\varepsilon_{k-1}}, 
\end{equation}
whence we establish the relationship between the convergence rates of $\{z_k\}$ and $\{\varepsilon_k\}$. \cref{1008eq:10} implies that $\{z_k\}$ converges $R$-superlinearly to $z^*$.  
Substituting $\theta = \frac{1}{2}$ into \cref{Eq:45}, and using $\varepsilon_k^{1-\theta} \geq \varepsilon_k^{1/2}$ for $\theta \in \left(\frac{1}{2}, 1\right)$, we can similarly prove the remaining convergence rate results for $\{z_k\}$ in (ii)--c and (ii)--d.
\end{proof}

\section{Numerical experiments}\label{sec:numerical}

In this section, we present numerical results of the proposed BPLADMM algorithm for Example~\ref{0425example:1} on the MNIST image classification dataset~\cite{lecun1998gradient}. The experiment is implemented in MATLAB R2021a on a standard laptop with an Intel i5-10210U CPU and 8 GB RAM.

\subsection{Parameter settings and closed-form solutions of subproblems}

This subsection presents the parameter setup and core subproblem solutions for the proposed BPLADMM algorithm.
The MNIST dataset is used for numerical validation, containing $28 \times 28$ grayscale handwritten digits across 10 classes, with 60,000 training and 10,000 test samples.
All images are normalized to $[0,1]$ and flattened into 784-dimensional vectors, while discrete labels are converted into one-hot vectors for the classification network. We randomly select 10,000 training and 2,500 test samples to balance computational cost and validity.

For implementation of BPLADMM, we rewrite the problem in Example~\ref{0425example:1} in standard form.
We adopt the standard softmax mapping \cite{bridle1989training}: 
\[\sigma^{(i)}(t)=\frac{e^{t_i}}{\sum_{p=1}^{K}e^{t_p}}.\]
Then, the nonlinear mapping term is formulated as $F^{(i)}(x) = \sigma^{(i)}(X^\top v_i + \hat b_i) \in \mathbb{R}^{K}$, in which $v_i$ is the feature vector of the $i$-th sample. Let 
\[F(x) = \big((F^{(1)}(x))^\top, \dots, (F^{(N)}(x))^\top\big)^\top,\, M = I_{NK},\] which implies $F(x)=u$.
In our experiments, we adopt a dynamically chosen auxiliary function $\phi_k(x)=\tfrac{1}{2}x^\top(\tau I-\nabla F(x_k)^\top \nabla F(x_k))x$
to obtain the closed-form solution of the $x$-subproblem. Since $\sigma$ is the softmax activation function, an appropriate choice of $\tau$ always guarantees that this $\phi(x)_k$ consistently satisfies Lemma~\ref{1122lem:1}, which is consistent with the variable-kernel conditions used in Algorithm~\ref{alg:1}. This choice also cancels the non-diagonal quadratic term generated by the linearized augmented-Lagrangian penalty and preserves a closed-form shrinkage update.
In our numerical experiments, the parameters are set as $\gamma = 10^{-4}$, $\rho = 0.9$, $\tau = 15$, with a maximum iteration number of $10^4$. 
Following \cite{chowdhury2024poissonian,wang2022minimizing}, we assume $\inf_k g(x_k) \geq \varepsilon > 0$ with $\mathcal{S} = \mathbb{R}^n$, which is numerically verifiable. 
The stopping criteria for BPLADMM are defined as
\begin{align*}
\err_1 &= \left\|\frac{g(x)\,\partial h(x) - h(x)\,\nabla g(x)}{[g(x)]^2} + \nabla F(x)^\top \lambda \right\|,\,
\err_2 = \left\|\nabla f(u) - \lambda\right\|,\\
\err_3 &= \left\|F(x_{k+1}) - u_{k+1}\right\|,\,
\err_4 = \max\Big(\frac{\|x_{k+1}-x_k\|}{\|x_k\|},\frac{\|u_{k+1}-u_k\|}{\|u_k\|},\frac{\|\lambda_{k+1}-\lambda_k\|}{\|\lambda_k\|}\Big),
\end{align*}
and the algorithm terminates when \text{Residual}:=$\max\left(\err_1,\err_2,\err_3,\err_4\right) \leq 0.01$.
The subproblems in \cref{1129Equ3.5} and \cref{Equ:10} admit closed‑form solutions. Let $H_k=\lambda_k+\rho(F(x_k)-u_k)$ and then the $x$‑subproblem has the following closed-form solution
\begin{align*}
x_{k+1} &= \underset{x}{\arg \min} \Bigl\{ c\gamma\|x\|_1 - \bigl\langle \gamma c^3\|x_k\|x_k ,\, x-x_k \bigr\rangle +\widetilde \psi(x,u_k,\lambda_k;x_k) + D_{\phi_k}(x,x_k) \Bigr\}\\[4pt]
&= \underset{x}{\arg\min} \left\{ c\gamma\|x\|_1 + \frac{\tau}{2} \left\| x - \left(x_k-\frac{\nabla F(x_k)^\top H_k-\gamma c^3\|x_k\|x_k}{\tau} \right)\right\|^2 \right\}\\[4pt]
&= \operatorname{shrink}\!\left( x_k - \frac{ \nabla F(x_k)^\top H_k - \gamma c^3\|x_k\|x_k }{\tau},\; \frac{c\gamma}{\tau} \right),
\end{align*}
where $\operatorname{shrink}(p,q)$ denotes the standard shrinkage operator \cite{combettes2005signal,wang2008new}.
The $u$-subproblem solution reads 
\[u_{k+1} = \underset{u}{\arg\min} \left\{ -\frac{1}{m}\sum_{i=1}^{m}\log(u_{i}^\top Y_{i}) + \frac{\rho}{2}\left\|F(x_{k+1}) - u + \frac{\lambda_k}{\rho}\right\|^2 \right\},\]
from Example~\ref{0425example:1}. 
Setting the gradient to zero and solving the ensuing scalar equation yields the following closed-form solution:
\[
    u_{i,j}=
    \begin{cases}
    \dfrac{1}{2}\left(\nu_{i,k_i} + \sqrt{\nu_{i,k_i}^2 + \frac{4}{m\rho}}\right), & j = k_i,\\[4pt]
    \nu_{i,j}, & j \neq k_i,
    \end{cases}
    \]
where $\nu_i$ denotes the $i$-th block of $F(x_{k+1}) + \lambda_k/\rho$, and the index $k_i$ is chosen such that $Y_{i,k_i} = 1$.
Note that the logarithm function requires a positive domain input. Fortunately, according to the closed-form solution of the $u$-subproblem, all components of $u$ can be guaranteed to stay positive during iteration by appropriately choosing the penalty parameter $\rho$, provided that the initial value $u_0$ is component-wise positive. This is because the update of u depends on $F(x_{k+1})+\lambda_k/\rho$, where the mapping $F(\cdot)$ is strictly positive. The positivity of $u$ can be further verified in numerical experiments.

\subsection{Experimental results}

In this section, we implement Algorithm~\ref{alg:1} and investigate its numerical performance. Owing to the complex and unique formulation of the target problem, no existing algorithms are suitable for direct comparison against our approach. For the BPLADMM algorithm, the initial points are generated by following MATLAB command: x0 = 0.001 * ones(10*(n+1)); u0 = 0.001 * ones(10*m); lambda0 = u0.

\begin{figure}[h]
\centering

\begin{minipage}[t]{0.48\textwidth}
\centering
\includegraphics[width=\textwidth,height=4.5cm]{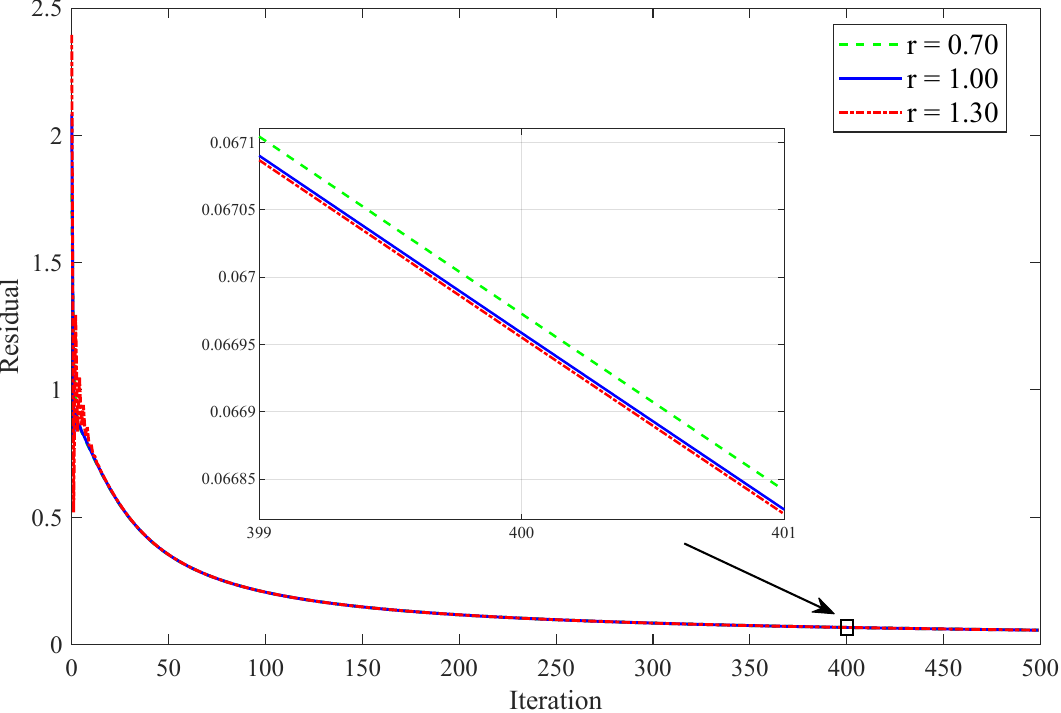}
\vspace{-0.5cm}
\caption{Residual versus iteration number under different $r$ values.}
\label{fig:1}
\end{minipage}
\hfill
\begin{minipage}[t]{0.48\textwidth}
\centering
\includegraphics[width=\textwidth,height=4.5cm]{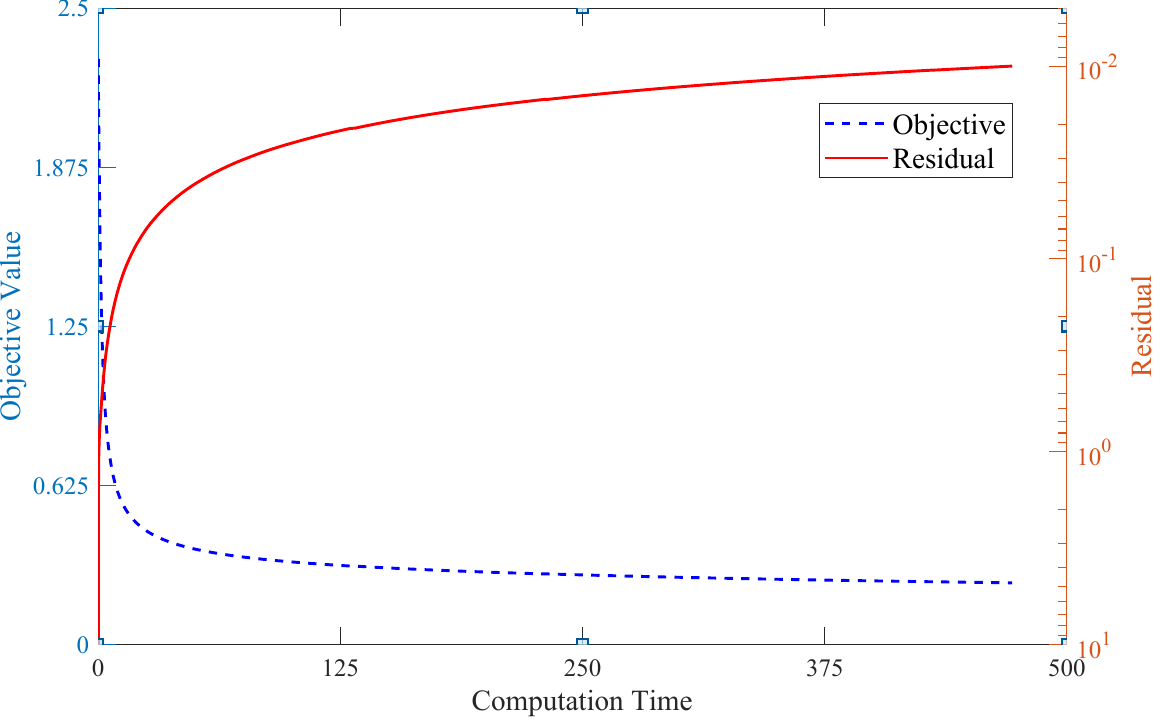}
\vspace{-0.5cm}
\caption{Convergence curves versus computation time.}
\label{fig:2}
\end{minipage}
\vspace{10pt} \\  
\begin{minipage}[t]{0.48\textwidth}
\centering
\includegraphics[width=\textwidth,height=4.5cm]{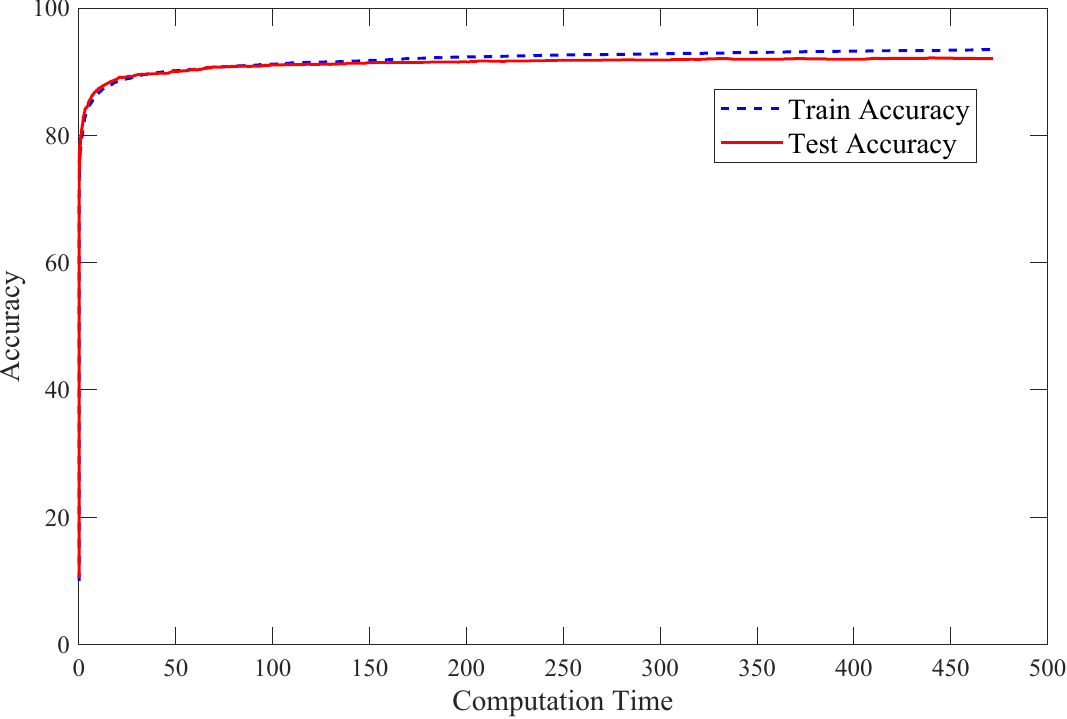}
\vspace{-0.5cm}
\caption{Classification accuracy versus computation time.}
\label{fig:3}
\end{minipage}
\hfill
\begin{minipage}[t]{0.48\textwidth}
\centering
\includegraphics[width=\textwidth,height=4.5cm]{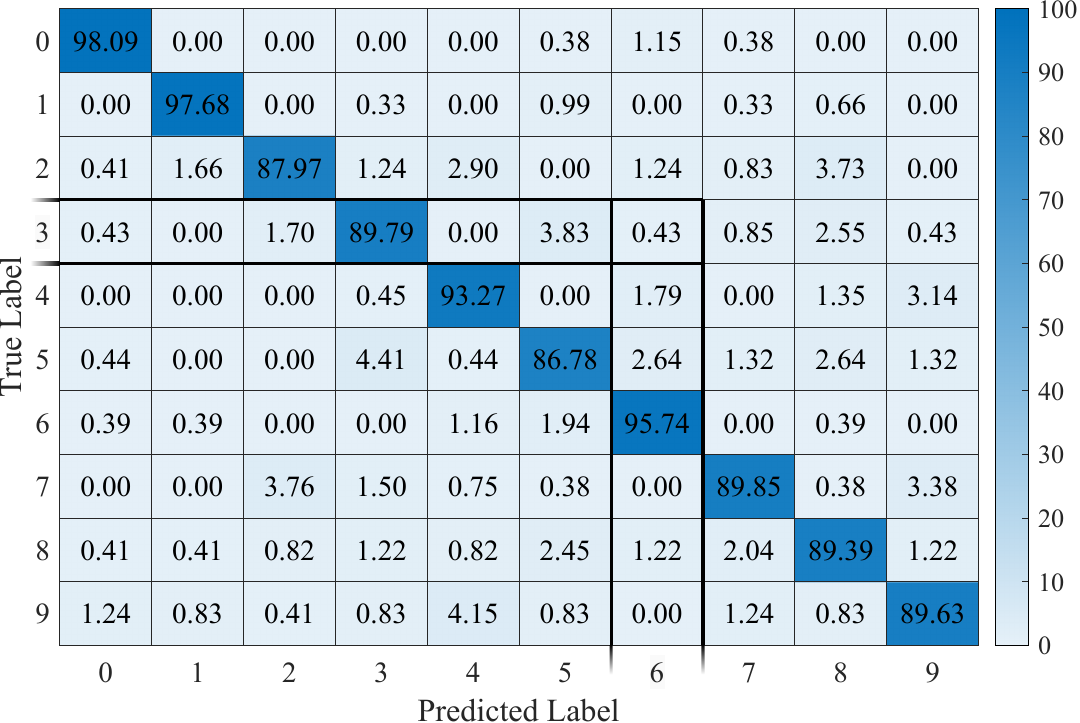}
\vspace{-0.5cm}
\caption{Confusion matrix on test dataset.}
\label{fig:4}
\end{minipage}
\end{figure}
Figure~\ref{fig:1} plots the convergence curves of residual versus iteration number over 500 iterations with three different parameters $r=0.70$, $r=1.00$, and $r=1.30$. The embedded zoom-in subplot presents the detailed local variations of residual near 400 iterations. $r=1.30$ outperforms the other two settings, achieving faster convergence and a smaller residual in the late stage.
This indicates that BPLADMM is flexible and effective over a wide range of $r$. Thus, we adopt $r=1.30$ for the subsequent experiment.
Figure~\ref{fig:2} plots the objective value and residual against computation time. The objective drops rapidly and stabilizes near 0.2422, while the residual reaches the target of $10^{-2}$ at 471.84 seconds.
Figure~\ref{fig:3} plots the training and test classification accuracy against computation time. Both training and test accuracy rise rapidly at the early stage and gradually saturate as computation time increases. The training accuracy is slightly higher than the test accuracy, while the two curves remain close throughout, indicating mild generalization gap and no severe overfitting.
Figure~\ref{fig:4} presents the normalized confusion matrix. Most test samples are correctly classified, with high diagonal values across all classes. The misclassification rates remain low, confirming the method’s effective discrimination among digits.
These results further validate the efficiency of the proposed BPLADMM.

\section{Conclusions}\label{sec:conclusions}

In this paper, we propose the BPLADMM for solving problem~\cref{Equ:1}. The proposed BPLADMM allows a wide range of multiplier step sizes, which enhances its flexibility in parameter tuning. For problem~\cref{Equ:1}, following \cite{zhou2025min}, we eliminate the fractional term in the objective function by introducing a new variable, thereby reformulating the fractional component as a non-fractional problem. The resulting Lyapunov function no longer contains fractional terms and enables easy-to-solve subproblems by linearizing the nonlinear part of the functional constraints within the augmented Lagrangian. We prove that any accumulation point of the sequence generated by BPLADMM is a lifted critical point of problem~\cref{Equ:1}. Moreover, under mild conditions, including the assumption that the constructed Lyapunov function satisfies the KL property with exponent $\theta \in [0,1)$, we establish the global convergence of the entire iterative sequence. Furthermore, we derive convergence rates for both the Lyapunov function sequence and the iterative sequence under HVP-EB condition and KL property; in particular, superlinear convergence is achieved when $\theta \in (0, \frac{1}{2})$. Numerical experiments confirm the effectiveness of the proposed method.

\end{document}